\documentclass[11pt]{article}

\usepackage[margin=0.97in]{geometry}
\usepackage{amsthm}
\usepackage{amsmath}
\usepackage{amssymb}
\usepackage{url}
\usepackage{paralist}
\usepackage[usenames]{color}
\usepackage{enumerate}

\usepackage{algorithm}
\usepackage{algpseudocode}
\usepackage[normalem]{ulem}

\usepackage{verbatim} 

\renewcommand{\phi}{\varphi}

\newcommand{\hide}[1]{ }

\renewcommand{\phi}{\varphi}

\DeclareMathOperator{\argmin}{argmin}

\newcommand{\eps}{\epsilon}

\DeclareMathOperator{\poly}{poly}

\newcommand{\Dist}{\Delta}
\newcommand{\SymDif}{\ominus}

\theoremstyle{plain}
\newtheorem{theorem}{Theorem}[section]

\newtheorem{lemma}{Lemma}[section]
\newtheorem{claim}[lemma]{Claim}
\newtheorem{fact}[lemma]{Fact}
\newtheorem{observation}[theorem]{Observation}

\newtheorem{corollary}[theorem]{Corollary}

\newtheorem{definition}{Definition}

\usepackage[normalem]{ulem}

\renewcommand{\include}{\input}

\newcommand{\dchange}[1]{{\color{blue}  #1}}
\newcommand{\nchange}[1]{{\color{green}  #1}}
\newcommand{\dcom}[1]{{\color{magenta}{[D: #1]}}}
\newcommand{\ncom}[1]{{\color{magenta}{[N: #1]}}}
\newcommand{\TODO}[1]{ {\color{red} TODO: #1 }}

\newcommand{\prevv}[1]{}

\newcommand{\calA}{\mathcal{A}}
\newcommand{\calB}{\mathcal{B}}
\newcommand{\calC}{\mathcal{C}}
\newcommand{\calF}{\mathcal{F}}
\newcommand{\calH}{\mathcal{H}}
\newcommand{\calJ}{\mathcal{J}}
\newcommand{\calK}{\mathcal{K}}
\newcommand{\calP}{\mathcal{P}}
\newcommand{\calS}{\mathcal{S}}
\newcommand{\calQ}{\mathcal{Q}}

\newcommand{\calV}{\mathcal{V}}

\newcommand{\calHPP}{\mathcal{HPP}}
\newcommand{\tcalHPP}{\mathcal{SHPP}}
\newcommand{\GPP}{\mathcal{GPP}}
\newcommand{\GPPWR}{\mathcal{GPP}_{\rm WR}}

\newcommand{\barC}{\bar{C}}

\newcommand{\hDelta}{{\widehat{\Delta}}}

\newcommand{\wtC}{{\widetilde{C}}}
\newcommand{\wtT}{{\widetilde{T}}}
\newcommand{\wtG}{{\widetilde{G}}}
\newcommand{\wtQ}{{\widetilde{Q}}}

\newcommand{\Phieps}{\Phi(\eps)}

\newcommand{\Phik}{\Phi^{k}}
\newcommand{\tPhik}{\widetilde{\Phi}^{k}}

\newcommand{\anc}{{a}}

\newcommand{\eqdef}{\stackrel{\rm def}{=}}

\renewcommand{\part}[2]{{#1}^{#2}}

\newcommand{\SUB}[1]{S^{ub}_{#1}} 
\newcommand{\SLB}[1]{S^{lb}_{#1}} 
\newcommand{\DUB}[1]{D^{ub}_{#1}} 
\newcommand{\DLB}[1]{D^{lb}_{#1}} 
\newcommand{\SP}[1]{\bar{S}_{#1}} 
\newcommand{\DP}[2]{\bar{D}_{#1, #2}} 
\newcommand{\TPD}[2]{\mathcal{TPD}^k_{#2}(#1)} 
\newcommand{\RPDist}[4]{L_1(\DP{#1}{#3}, \DP{#2}{#4})} 
\newcommand{\PDDist}[2]{L_1(\DLB{#1}, \DLB{#2})} 
\newcommand{\Lone}[2]{L_1(#1,#2)} 
\newcommand{\Sd}{\mathcal{D}} 
\definecolor{ourgray}{gray}{0.6}
\newcommand{\maxX}{X}

\newcommand{\calGPP}{\mathcal{GPP}}

\newcommand{\ind}{{\rm ind}}

\def\fullver{1}

\begin{document}

\begin{titlepage}
\title{On Efficient Distance Approximation for Graph Properties}
\author{
    Nimrod Fiat \\
    Tel Aviv University \\
         {\tt nimrod.fiat@gmail.com}
\and
 Dana Ron\thanks{Supported by the Israel Science Foundation (grant number~1146/18) and the Kadar-family award.} \\
 Tel-Aviv University,\\
 {\tt danaron@tau.ac.il}
 }
	\maketitle

	\begin{abstract}
		
		A \emph{distance-approximation\/} algorithm for a graph property $\calP$ in the adjacency-matrix model is given an approximation parameter $\eps \in (0,1)$ and query access to the adjacency matrix of a graph $G=(V,E)$. It is required to output an estimate of the \emph{distance} between $G$ and the closest graph $G'=(V,E')$ that satisfies $\calP$, where the distance between graphs is the size of the symmetric difference between their edge sets, normalized by $|V|^2$.

		In this work we
introduce \emph{property covers}, as a framework for using
distance-approximation algorithms for ``simple'' properties to design distance-approximation 
algorithms for more ``complex'' properties. 
Applying this framework we present distance-approximation algorithms with
$\poly(1/\eps)$ query complexity for
induced $P_3$-freeness, induced $P_4$-freeness, and Chordality. For induced $C_4$-freeness our algorithm has query complexity $\exp(\poly(1/\eps))$. These complexities essentially match the corresponding known results for testing these properties and provide an exponential improvement on previously known results.
	\end{abstract} 

\thispagestyle{empty}

\end{titlepage}

\ifnum\fullver=1	
	\tableofcontents
\setcounter{page}{1}
	\newpage
\fi

\section{Introduction}\label{sec:intro}


A \emph{property} is simply a set of objects (e.g., graphs). A
\emph{distance-approximation algorithm} for a property $\calP$ and a prespecified distance measure is an algorithm that approximates the distance between a given object and the closest object in $\calP$.
A related decision task, known as {\em property testing\/}, is to distinguish between objects that belong to $\calP$ and objects that are relatively far from $\calP$ (i.e., are far from any object in $\calP$).
For both tasks, algorithms are given query access to the input object and are allowed a small failure probability. The goal is to design algorithms that perform as few queries as possible.
While there are contexts in which knowing whether an object has a certain property or is far from having it is sufficient for our needs, in others it is actually important to have a good estimate of the distance.
Hence, we are interested in studying the stronger notion of distance approximation (which is typically more challenging than property testing).\footnote{Note that if an object is close to having a property, then there are no requirements from the testing algorithm. Hence, we cannot use a testing algorithm as a black box for distance approximation.}


The objects we consider are graphs. We work in the {\em adjacency-matrix model\/}~\cite{GGR} (also known as the {\em dense-graph model\/}), in which the algorithm can query the adjacency matrix of the tested graph $G = (V,E)$. That is, for any pair of vertices $u,v \in V$, it can determine  whether or not $(u,v) \in E$.
In this model,
the distance between a graph $G$ and a property $\calP$, denoted by
$\Delta(G,\calP)$, is the minimum number of edges that should be added to/removed  from $G$ so as to obtain a graph in $\calP$, normalized by $|V|^2$ (the size of the adjacency matrix).
Given an approximation parameter $\eps$, a distance-approximation algorithm for property $\calP$ is required to output an estimate $\hDelta$ such that $|\hDelta-\Delta(G,\calP)| \leq \eps$ with probability at least $2/3$. A property testing algorithm for $\calP$ is required to distinguish with probability at least $2/3$ between the case that $G\in \calP$ and the case that
$\Delta(G,\calP)>\eps$.

\paragraph{General results for testability and approximability.}
\sloppy
Building on the Regularity Lemma of Szeme\'redi~\cite{Sz}, Alon et al.~\cite{AFNS}
gave a characterization of all graph properties that can be tested with query complexity that is only a function of $\eps$ (and has no dependence on $n = |V|$). Such properties are 
often referred to as \emph{testable}.
Independently, Borgs et al.~\cite{BCLSSV} obtained
an analytic characterization of 
testable properties through the theory
of graph limits.

Turning to distance-approximation, Fischer and Newman~\cite{FN} showed
that every testable property
has a distance-approximation algorithm whose query complexity is at most
``Wowzer'' (a composition of Tower functions) 
of $\poly(1/\eps)$.\footnote{A related result regarding distance approximation (based on graph limits) appears in~\cite{BCLSV}, but does not give explicit bounds on the query complexity.}
Alon, Shapira and Sudakov~\cite{ASS} improve on this result for monotone properties (properties
inherited by subgraphs), giving a distance-approximation algorithm for monotone properties with query complexity that is at most a tower of height $\poly(1/\eps)$.
The result of Hoppen et al.~\cite{HKLLS16}
combined with the result of Fox~\cite{Fox-removal}, 
implies that
a tower of height $\poly(\log(1/\eps))$ suffices (for monotone properties).
The follow-up work of Hoppen et al.~\cite{HKLLS17} combined with the result of Conlon and Fox~\cite{CF-removal}
implies that hereditary properties (properties
inherited by induced subgraphs)
have distance-approximation algorithms with query complexity that is at most
a tower of height $\poly(1/\eps)$.
In 
some cases, which we discuss below, the result of~\cite{HKLLS17} 
implies significantly more efficient distance-approximation algorithms, though never better than $\exp(\poly(1/\eps))$.

While these results are 
general,
in many cases they are far from optimal.
A natural question that arises (and is explicitly stated in the aforementioned papers) is for which properties are there testing/distance-approximation  algorithms whose query complexity is significantly smaller, and in particular, polynomial in $1/\eps$.

In what follows we shortly survey the known results relating to the above question. While there are quite a few known results for testing graph properties with $\poly(1/\eps)$ query complexity, 
relatively
little is known for distance-approximation algorithms. Indeed, in this work we set out to remedy this situation, by (almost) closing the ``knowledge gap'' between
$\poly(1/\eps)$-query testing and distance approximation.
Our work and the works discussed above
leave as an open question
whether the complexity of distance approximation in the adjacency-matrix model is always polynomially related to the complexity of testing.


\paragraph{Results for \boldmath{$\poly(1/\eps)$}-testability.}
\sloppy
The first $\poly(1/\eps)$-testable properties were presented  in~\cite{GGR}. These include Bipartiteness, and more generally, $k$-colorability, $\rho$-clique and $\rho$-cut (having a clique of size $\rho n$, and having a cut of size $\rho n^2$, respectively). Furthermore, Goldreich et al.~\cite{GGR} defined a family of
\emph{General Partition Properties}
(which the aforementioned properties belong to),
and proved that every property in this family is $\poly(1/\eps)$-testable.\footnote{For each  property in this family, a graph has the property if its vertex set can be partitioned into $k$ parts in a manner satisfying certain constraints on the number of edges within parts and between parts -- A formal definition 
	appears in 
	Appendix~\ref{app:gpp}.}
An extension of this result, which covers some additional partition properties, appears in~\cite{NR18}.

Another type of graph properties that have been studied in the context of property testing are those defined by being subgraph-free of a small fixed graph.
Alon~\cite{Alon-iff-bipartite} proved that (non-induced) $H$-freeness (for a constant-size graph $H$) is $\poly(1/\eps)$-testable if and only if $H$ is bipartite.\footnote{The super-polynomial lower bound for non-bipartite $H$ was proved in~\cite{Alon-iff-bipartite} for one-sided error algorithms, and this result was extended to two-sided error algorithms in~\cite{AS-direct}. }
For induced $H$-freeness, Alon and Shapira proved that for any graph $H$ except $P_2$, $P_3$, $P_4$, $C_4$ (and their complements),\footnote{For an integer $\ell$, $P_\ell$ is the path over $\ell$ vertices, and $C_\ell$ is the cycle over $\ell$ vertices. It is also common to use $P_\ell$ to denote the path with $\ell$ \emph{edges}.} induced $H$-freeness is {\em not\/} $\poly(1/\eps)$ testable. On the positive side, in addition to induced $P_2$-freeness (a single edge), which is clearly testable with query complexity $O(1/\eps)$, both induced $P_3$-freeness and induced $P_4$-freeness are $\poly(1/\eps)$-testable (\cite{AS-easily-induced} and~\cite{AF15}, respectively).

Gishboliner and Shapira~\cite{GS17} show 
$\poly(1/\eps)$-testability for a family of graph properties that includes both induced $P_4$ freeness and
induced $P_3$ freeness.\footnote{A full description of the family is
	somewhat involved, and hence we do not elaborate on it here. We note that it  captures induced freeness from any finite family of graphs that includes a split graph, a bipartite graph and a co-bipartite graph.}
Recently,
Gishboliner and Shapira~\cite{GS19} proved that induced $C_4$-freeness is testable with query complexity $\exp(\poly(1/\eps))$. While this still leaves open the question of $\poly(1/\eps)$-testability of induced $C_4$-freeness, it is a 
significant improvement over the best previously known upper bound (of Tower complexity). They also showed chordality (a subcase of $C_4$-freeness) is testable with query complexity $\exp(\poly(1/\eps))$.
This was subsequently improved by De~Joannis~de~Verclos~\cite{verclos2019chordal}, who showed that chordality is testable with query complexity $\poly(1/\eps)$.

\paragraph{Distance approximation.}
While the focus of~\cite{GGR} was on property testing, they also gave one distance-approximation algorithm.
The algorithm
approximates the size of a maximum $k$-cut (the maximum number of edges crossing a $k$-cut), and hence implies a distance-approximation algorithm for $\rho$-$k$-cut
(and therefore for $k$-colorability as well). Its query complexity is polynomial in $1/\eps$.

The result of~\cite{HKLLS17}, combined with the analysis regarding testability discussed above, 
implies distance-approximation algorithms with query complexity 
$\exp(\poly(1/\eps))$ for induced $P_3$-freeness and induced $P_4$-freeness, and double exponential for chordality and $C_4$-freeness.


We 
observe that distance approximation to (non-induced) $H$-freeness for bipartite graphs $H$ can be easily performed using $O(1/\eps^2)$ queries by simply estimating the number of edges in the graph. This is the case because for every bipartite graph $H$ over $t$ vertices, every graph with $n^{2-\Omega(\frac{1}{t})}$ edges contains $H$ as a subgraph~\cite{Alon-iff-bipartite}.
Also note that all query complexity lower bounds  for  testing (such as the one for triangle freeness~\cite{Alon-iff-bipartite,AS-direct}), are also lower bounds for distance approximation.

\subsection{Our results}\label{subsec:our-res}


In this work
we describe a framework for designing distance-approximation algorithms, and use it to
obtain efficient distance-approximation algorithms for 
properties that have known efficient property testing algorithms (surveyed above).
The only known property-testing result
with $\poly(1/\eps)$ query complexity
for which we do not provide a corresponding distance-approximation result is the one
presented in~\cite{GS17}.
We also present a distance-approximation algorithm
for $C_4$-freeness with $\exp(\poly(1/\eps))$ query complexity
(improving on the previously best known result of double-exponential complexity~\cite{HKLLS17} and matching the complexity of the testing result for this property~\cite{GS19}).


\subsubsection{A general framework}
\label{subsubsec:intro-framework}

Our distance-approximation algorithms are derived using a common framework that we introduce.
We are hopeful that this framework can be applied to derive additional new results.
In fact, as we discuss in more detail in Section~\ref{subsec:disc}, in retrospect,
the use of ``covers'' (see Definition~\ref{def:cover} stated next) is implicit in some of the previous works
(though the covers used, and the way they were used, implied complexity at least $\exp(\poly(1/\eps))$).

\begin{definition}\label{def:cover}
	Let $\calP$ be a graph property,  $\calF$  a family of graph properties, and  $\eps>0$.
	We say that $\calF$ is an $\eps${\sf-cover} for $\calP$
	if the following conditions holds:
	\begin{enumerate}
		\item\label{it:G-in-P} For each $G \in \calF$, there exists $\calP'\in \calF$ such that $\Delta(G, \calP') \leq \eps/2$.
		\item\label{it:Gprime-in-union} For each $G' \in \underset{\calP'\in\calF}{\bigcup} \{\calP'\}$, $\Delta(G', \calP) \leq \eps/2$.
	\end{enumerate}
\end{definition}

The high-level idea for the framework is the following.
Let $\calP$ be a graph property for which we would like to design a distance-approximation algorithm.
In order to apply the framework, we show how, for any given $\eps>0$ we can find an $\eps$-cover for $\calP$ by a family, $\calF(\eps)$, 
of
properties, which we refer to as \emph{basic} properties. These properties are basic in the sense that we have
efficient distance-approximation algorithms for them.  Moreover, these algorithms are non-adaptive and $|\calF(\eps)|$ is not too large.
This allows us to use the same queries to
obtain an estimate of $\Delta(G,\calP')$ for all $\calP'\in \calF(\eps)$.
It follows from Definition~\ref{def:cover} (and we prove this formally in Theorem~\ref{thm:cover}), that if we take
the minimum estimate obtained (over all $\calP'\in \calF(\eps)$), then we get a good approximation for $\Delta(G,\calP)$.


Hence,
in order to apply the framework to a graph property $\calP$ 
we must find a suitable cover consisting of properties that have an efficient distance-approximation algorithm.
Recall that in the context of testing, there are two types of properties for which there are known (non-trivial) testing algorithms with query complexity $\poly(1/\eps)$.
The first are partition properties (as defined in~\cite{GGR} (and extended in~\cite{NR18})) and the second are those defined by forbidden induced subgraphs.
For the former, the cover for each property is a  ``natural'' one, consisting of a subset of all partition properties, 
(details are given in Appendix~\ref{app:gpp}).
On the other hand, for all the latter
the covers are perhaps  more surprising, as they are seemingly unrelated to subgraph freeness.  Rather, they are subfamilies of a family of partition properties (introduced in~\cite{NR18}), which we refer to as  \emph{semi-homogeneous partition properties}.

In the next subsection we define this family, and in the following ones we shortly discuss each of our specific distance-approximation results. Our emphasis is on the way we construct a cover for each property, and the results for induced subgraph freeness for $P_3,P_4,C_4$ and Chordality are presented from simplest to more complex.

\subsubsection{Semi-homogeneous partition properties}\label{subsubsec:intro-SHPP}
Each Semi-homogeneous partition property
is defined by an integer $k$ and a symmetric function $\phi: [k]\times [k] \to \{0,1,\bot\}$, called the {\em partition function\/} (where $[k] \eqdef \{1,\dots,k\}$).
A graph $G = (V,E)$ has the corresponding property $\calP_\phi$ if its vertex set $V$ can be partitioned into $k$ parts $(V_1,\dots,V_k)$ such that the edge densities within and between parts are as indicated by $\phi$ (where $\bot$ stands for ``don't care''). To be precise,
For every $i,j \in [k]$,
if $\phi(i,j) = 1$,
then $G$ contains all edges with one endpoint in $V_i$ and one endpoint in $V_j$ (excluding self-loops in the case of $i=j$), and if $\phi(i,j)=0$, then there are no such edges. We say in such a case that the partition $(V_1,\dots,V_k)$
is a \emph{witness} partition for $\phi$.

We denote the family of semi-homogeneous partition properties (for a given number of parts $k$) by  $\tcalHPP^k$.
For example, $k$-colorability is in $\tcalHPP^k$.
The next lemma can be shown to follow from~\cite{AE02} (see Appendix~\ref{app:shpp-csp}).
\begin{lemma}\label{lem:shpp-are-ee}
	There exists an algorithm that, given $k$, $\phi: [k]\times [k] \to \{0,1,\bot\}$, $\eps,\delta >0$ and query access to a graph $G$,
	takes a sample of
	$\poly(1/\eps,\log k,\log(1/\delta))$
	vertices, selected uniformly, independently at random,
	queries the subgraph induced by the sample,
	and outputs an estimate $\hDelta$ such that with probability at least $1-\delta$ satisfies
	$|\hDelta - \Delta(G,\calP_\phi)| \leq \eps$.
\end{lemma}
We note that the fact that the dependence on the number of parts, $k$, is only polylogarithmic, is crucial for some of our applications.

\subsubsection{Induced $P_3$-freeness}\label{subsubsec:intro-P3}
The first and simplest application of our framework is to induced $P_3$-freeness, for which we prove the
following theorem.
\begin{theorem} \label{thm:p3_is_EE}
	There exists a distance-approximation algorithm 
	for induced $P_3$-freeness whose query complexity is $\poly(1/\eps)$.
\end{theorem}
In order to prove Theorem~\ref{thm:p3_is_EE} we define a family $\calF \subset \tcalHPP^k$ where $k = O(1/\eps)$ and $|\calF|=1$. In fact, the single property $\calP_\phi$ in $\calF$ belongs to the more restricted class of \emph{homogeneous partition properties}, where the range of the partition function $\phi$ is $\{0,1\}$ (rather than $\{0,1,\bot\}$).
This family contains all graphs that are a union of $O(1/\eps)$ cliques.
Since graphs that are induced $P_3$-free are known to be characterized by being a union of (any number of) cliques,
it quite easily follows that induced $P_3$-freeness is covered by this singleton family of homogeneous partition properties.
(We comment that this characterization was also used for efficient testing of induced $P_3$-freeness~\cite{AS-easily-induced}.)

\subsubsection{Induced $P_4$-freeness}\label{subsubsec:intro-P4}
Our next application is to induced $P_4$-freeness.

\begin{theorem} \label{thm:p4_is_EE}
	There exists a distance-approximation algorithm 
	for induced $P_4$-freeness whose query complexity is $\poly(1/\eps)$.
\end{theorem}

In order to prove Theorem~\ref{thm:p4_is_EE}, we build on a known characterization of induced $P_4$-free graphs.
For every such graph $G$, there exists an auxiliary tree, which we denote by $T_G$, whose leaves correspond to vertices of $G$, and whose internal nodes correspond to cuts in $G$. We show how by performing certain pruning and contraction operations on $T_G$, we can obtain a tree $T'$ for which the following holds. The tree $T'$ can be used to define a
homogeneous partition function $\phi$ with $k = O(1/\eps)$ parts, such that $G$ is close to $\calP_\phi$.
Furthermore, $\calP_\phi$ belongs to a family $\calF \subset \tcalHPP^k$ of size $\exp(\poly(1/\eps))$
such that every graph $G' \in \calF$ is induced $P_4$-free. In other words, $\calF \subset \tcalHPP^k$ is a cover for induced $P_4$-freeness (with $k = O(1/\eps)$ and $|\calF| \leq \exp(\poly(1/\eps))$, and we can
derive Theorem~\ref{thm:p4_is_EE}.\footnote{We note in passing that while the exponent of the polynomial (in $1/\eps$) that bounds the query complexity of our algorithm is quite high, it is actually lower than the corresponding exponent for property testing (with one-sided error)~\cite{AF15}. (Indeed both in~\cite{AF15} and here no attempt was made to optimize the exponent.)}

\subsubsection{Induced $C_4$-freeness}\label{subsubsec:intro-C4}
Recall that the best known testing algorithm for $C_4$-freeness~\cite{GS19} has query complexity $\exp(\poly(1/\eps))$.
We show that distance approximation can be performed with similar complexity.
\begin{theorem} \label{thm:c4_is_EE}
	There exists a distance-approximation algorithm 
	for induced $C_4$-freeness whose query complexity is 
	$\exp(\poly(1/\eps))$.
\end{theorem}

To prove Theorem~\ref{thm:c4_is_EE} we build on a lemma concerning induced $C_4$-free graphs, which follows from~\cite{GS19}.
The lemma shows that every induced $C_4$-free graph $G=(V,E)$ is close to another induced $C_4$-free graph $G'=(V,E')$ with a useful property. Specifically, the vertices of $V$ can be partitioned into an independent set, $I$, and a collection of cliques, $Q_1,\dots,Q_t$, such that between every two cliques there are either no edges, or all possible edges.
Furthermore, the neighbors of every vertex $y\in I$ induce a clique in $G'$.

Using this lemma we define a family $\calF \subset \tcalHPP^k$, where as opposed to the case of induced $P_3$-freeness and induced $P_4$-freeness, the partition properties in $\calF$ are not homogeneous properties but only semi-homogeneous. Roughly speaking, when attempting to define a family of partition properties $\calF$ such that $G'$ (and hence $G$) is close to a graph having some property in $\calF$, we need to allow for  non-homogeneity between parts that refine the independent set $I$ and parts corresponding to the cliques $\{Q_i\}_{i=1}^t$. Since the number of parts $k$ as well as the size of $\calF$ are double-exponential in $\poly(1/\eps)$, the complexity of the resulting distance-approximation algorithm (as stated in Theorem~\ref{thm:c4_is_EE})  is $\exp(\poly(1/\eps))$.

\subsubsection{Chordality}\label{subsubsec:intro-chordal}
Recall that a graph is \emph{chordal} if every cycle of length greater than three in the graph contains a chord (i.e., an edge between two non-consecutive vertices on the cycle). Equivalently, a graph is chordal if it contains no induced $C_k$ for $k>3$
(so that in particular it is induced $C_4$-free).
Chordal graphs have been studied extensively in the context of optimization problems (see for example the survey~\cite{ChordalSurvey}.

As stated previously, in~\cite{GS19} it was shown that chordality can be tested with query complexity $\exp(\poly(1/\eps))$, and this was improved to $\poly(1/\eps)$ in~\cite{verclos2019chordal}.
\begin{theorem} \label{thm:chordaily_is_EE}
	There exists a distance-approximation algorithm
	for chordality whose query complexity is $\poly(1/\eps)$.
\end{theorem}


Our starting point is a well known characterization of chordal graphs, which states that a graph is chordal if and only if its maximal cliques can be arranged in a \emph{clique tree} (see Definition~\ref{def:clique-tree}).
One central ingredient in the proof of Theorem~\ref{thm:chordaily_is_EE} is showing that, roughly speaking, every chordal graph  $G=(V,E)$ is close to another chordal graph, $G'=(V,E')$, where $G'$ has a \emph{small} clique tree. (More precisely, $V$ can be partitioned into two parts, $X$ and $Y$ such that the subgraph of $G'$ induced by $X$ has a  clique tree of size $\poly(1/\eps)$ while the subgraph induced by $Y$ is empty, and the neighbors of every vertex in $Y$ form a clique.)
On a high-level, we show that a small clique tree can be obtained by removing edges from $G$ in a manner that ``shortens'' long paths in the tree.
A second main ingredient is the definition of a subset of semi-homogeneous partition properties that covers chordality,
based on small clique trees.
To obtain such a subset we consider all possible clique trees of size $\poly(1/\eps)$, and let the parts of the partitions in the cover be intersections of subsets of maximal cliques. We show how this ensures homogeneity between the parts.
(More precisely, these are the parts that correspond to partitions of $X$ -- the parts that correspond to a partitions of $Y$ are defined based on their neighbors in $X$.)
We hence get that chordality is $\eps$-covered by a family of size $\exp(\poly(1/\eps))$ of properties in $\tcalHPP^{k}$
for $k= \exp(\poly(1/\eps))$, from which Theorem~\ref{thm:chordaily_is_EE} follows.

\subsection{Discussion: on the benefits of non-regular partitions}\label{subsec:disc}
As noted at the start of Section~\ref{subsubsec:intro-framework},
all previous general results for distance approximation, namely, for all testable properties~\cite{FN}, monotone properties~\cite{ASS,HKLLS16} and hereditary properties~\cite{HKLLS17}
can be viewed as implicitly applying our covering framework.
The differences between the applications (which affect the resulting query complexities) are in the size of the covering family and in the  algorithm used to estimate the distance to members of the family. In all cases, the covering family is defined by regular partitions (either Szemeredi regular partitions~\cite{Sz} in~\cite{FN,ASS}, or Frieze-Kannan (``weak'') regular partitions~\cite{FK} in~\cite{HKLLS16,HKLLS17}). 
Roughly speaking, a regular partition of a graph is an equipartition of the  vertices, where the edge densities between pairs of subsets of parts are similar to the edge densities between the corresponding parts.
The two types of regular partitions differ in the precise way that similarity (regularity) is quantified.
The number of parts in the partitions defining the covering family as well as the size of the family depend on the type of regular partitions used, as well as on a regularity parameter $\gamma$.
 In turn, the setting of this parameter is dependent on the property to be covered, and is always at most $\eps$.

Furthermore, in the aforementioned works, the algorithm for approximating the distance of the given graph $G$ to
properties in the covering family works by constructing an approximate representation of a regular partition of $G$ (or possibly a set of approximations). 
The ``universality'' of regular partitions ensures that every graph has a regular partition
(of both aforementioned types). 
It has been shown by Conlon and Fox~\cite{CF11}, that even for the (weaker) Frieze-Kannan regular partitions,
some input graphs only have regular partitions 
of size $\exp(1/\gamma)$. Since (regardless of the property covered) $1/\gamma = \Omega(1/\eps)$ (and the distance-approximation algorithm must work for all graphs), the resulting query complexity is at least $\exp(1/\eps)$.

We are able to reduce the query complexity exponentially for the properties discussed in Sections~\ref{subsubsec:intro-P3}--\ref{subsubsec:intro-chordal} (and in particular obtain polynomial query complexity) by diverging from the above in two (related) ways.
First, the covering families we use are not defined by regular partitions, but rather by the simpler semi-homogeneous partitions. Second, the distance-approximation algorithm we use to semi-homogeneous partition properties does not attempt to construct a representation of a partition of the input graph.
This allows to obtain efficient distance-approximation algorithms
for properties that have small and simple covers
even though some input graphs are complex.
Hence, by giving up on the ``universality'' of regular partitions, and using the simpler semi-homogeneous partitions, we are able use smaller covers, and our dependence on the number of parts in these partitions (which may be $\exp(\poly(1/\eps))$ is logarithmic (rather than at least linear).

\subsection{Other related work on distance approximation}
\label{subsec:rel-work}
Distance approximation was  first
explicitly studied in~\cite{PRR} together with the closely related notion of
{\em tolerant testing\/}.\footnote{A tolerant-testing algorithm
is required  to accept objects that are $\eps_{1}$-close to having a given property $\calP$
and reject objects that are $\eps_{2}$-far from having property $\calP$,
for $0 \leq \eps_{1} < \eps_{2} \leq 1$.
Standard property testing refers to the special case of $\eps_1=0$.}
In~\cite{PRR} it is observed that some earlier works
imply results
for distance approximation.
In particular this
includes the aforementioned result for $\rho$-$k$-cut (and $k$-colorability)~\cite{GGR},
connectivity of sparse graphs~\cite{CRT}, edit
distance between strings~\cite{BEKMRRS}
and $\ell_1$-distance between distributions~\cite{BFRSW}.
The new results obtained in~\cite{PRR}
 were for monotonicity
 of functions $f: [n]\to R$, and clusterability of a set of
 points.  The first result was later improved in~\cite{ACCL}
and extended to higher dimensions in~\cite{FR}.

In~\cite{FF} it is shown that there are properties
of Boolean functions for which there exists a testing algorithm
whose complexity depends only on $\eps$ yet there is no such tolerant
testing algorithm. In contrast, as already noted, in~\cite{FN} it is shown
that {\em every\/} graph property that has a testing algorithm
in the adjacency-matrix model whose complexity is only a function of $\eps$,
has a distance-approximation algorithm
whose complexity is only a function of $\eps$.
Distance approximation in sparse graphs  is studied for a variety of properties (such as $k$-connectivity) in~\cite{MR} and in~\cite{CGR},
and there is a recent work on tolerant testing of arboricity in sparse graphs~\cite{ELR}.
Guruswami and Rudra~\cite{GuRu} present tolerant testing algorithms for
several constructions of locally testable codes, and
Kopparty and Saraf~\cite{KoSa} study tolerant linearity testing
under general distributions and its connection to locally testable
codes. Tolerant testing of image properties is studied in~\cite{BMR}, and tolerant junta testing of juntas in~\cite{BCELR}.

\section{Preliminaries}\label{sec:prel}

	\begin{definition}\label{def:dist}
		Two $n$-vertex graphs $G = (V,E)$ and $G'=(V,E')$ are at {\sf distance} ${\rho}$
if
$|E\setminus E'| + |E'\setminus E| = {\rho} n^2$.
We denote the distance between $G$ and $G'$ by $\Dist(G,G')$.

For a graph $G$ and a graph property $\calP$, the {\sf distance of $G$ to $\calP$}, denoted $\Dist(G, \calP)$, is the minimal distance between $G$ and any graph $G'\in \calP$.
	\end{definition}




	
	\begin{definition}\label{def:dist-approx-alg-add}
		We say that $\calA$ is a {\sf distance-approximation algorithm (with an additive error)} for graph property $\calP$ if the following holds for every graph $G = (V,E)$ and every $\eps\in (0,1)$.
Given $\eps$ 
as input and the ability to perform queries of the form: ``is $(u,v) \in E$'' for
$u,v \in V$, the algorithm $\calA$
 returns  an estimate $\hDelta$, such that
$|\hDelta -\Delta(G, \calP)| \leq  \eps$
  with probability at least 
  $2/3$.
	\end{definition}




For a graph $G = (V,E)$ and a subset $U \subseteq V$, we use the notation $G[U]$ for the subgraph of $G$ that is induced by $U$.
For a vertex $v\in V$ we use $N_G(v)$ to denote the set of neighbors of $v$ in $G$.
When we refer to a clique $Q$ in a graph $G=(V,E)$, we mean that $Q$ is a subset of $V$ such that the subgraph of $G$ induced by $Q$ is a clique.


	

	\begin{definition}\label{def:cuts}
For a graph $G = (V,E)$ and
a pair of subsets 	$U,U' \subseteq V$, we say that the cut $(U,U')$ is
{\sf empty} (in $G$)  if $(u,u') \notin E$ for every $u\in U, u' \in U'$.
We say that it is {\sf complete} if $(u,u') \in E$ for every $u\in U, u' \in U'$, $u \neq u'$.
If $(U,U')$ is either empty or completely, then we say that it is {\sf homogeneous}.
	\end{definition}
In Definition~\ref{def:cuts} we slightly abuse the notion of a cut to also include the case of $U' = U$
(so that an empty cut $(U,U)$ is an independent set and a complete cut $(U,U)$ is a clique).



\ifnum\fullver=1
We next present a central definition with many of the notions (and notations) used throughout this work.
(Some of these notions already appeared in the introduction.)
\begin{definition}\label{def:partitions}
For an integer $k$, a {\sf  homogeneous $k$-partition function} 
$\phi$ is a (symmetric) function from $[k]\times [k]$ to $ \{0,1\}$. If the range of $\phi$ is $ \{0,1,\bot\}$, then it is a {\sf semi-homogeneous partition function}.

We say that a graph $G = (V,E)$ {\sf obeys} $\phi$ if there exists a partition $(V_1,\dots,V_k)$ of $V$ such that
for every $i,j$ such that $V_i\neq \emptyset$ and $V_j \neq \emptyset$, we have that
$(V_i,V_j)$ is empty if and only if $\phi(i,j) = 0$ and $(V_i,V_j)$ is complete if and only if $\phi(i,j)=1$.
(If $\phi(i,j) = \bot$, then there is no restriction on the cut $(V_i,V_j)$.)
We refer to the partition $(V_1,\dots,V_k)$ as a {\sf witness} partition for $\phi$, or just as a
{\sf $\phi$-partition} of $G$.
For each $i \in [k]$ and $v\in V_i$, we say that $v$ is {\sf assigned} by the partition to {\sf part} $i$ of $\phi$.



Let $\Phik$ denote the set of all homogenous $k'$-partition functions for $1 \leq k' \leq k$, and let
$\tPhik$ denote the set of all semi-homogenous $k'$-partition functions for $1 \leq k' \leq k$.

For each (semi-)homogeneous partition function $\phi$, the corresponding (semi-)homogenous partition property, $\calP_\phi$, contains all graphs $G$ that obey $\phi$.
Let $\calHPP^k \eqdef \{\calP_\phi: \phi \in \Phik\}$, and $\tcalHPP^k \eqdef \{\calP_\phi: \phi \in \tPhik\}$.
\end{definition}

\else
Recall that in the introduction (Section~\ref{subsubsec:intro-SHPP}) we defined semi-homogeneous partition properties (based on semi-homogeneous partition functions $\phi:[k]\times[k] \to  \{0,1,\bot\}$). We next introduce several more related notions (and notations) used throughout this work.
We say that a partition function is \emph{homogeneous} if its range is $\{0,1\}$.
Let $\Phik$ denote the set of all homogenous $k'$-partition functions for $1 \leq k' \leq k$, and let
$\tPhik$ denote the set of all semi-homogenous $k'$-partition functions for $1 \leq k' \leq k$.
Recall that $\tcalHPP^k \eqdef \{\calP_\phi: \phi \in \tPhik\}$, and
let $\calHPP^k \eqdef \{\calP_\phi: \phi \in \Phik\}$,
We say that a graph $G = (V,E)$ {\sf obeys} $\phi$ if it belongs to $\calP_\phi$.
\fi

\section{Distance approximation by covering}
\label{sec:cover}

Recall the that definition of an $\eps$-cover (Definition~\ref{def:cover}) was given in the introduction.

	The following theorem is the base of our framework, and shows how to design distance-approximation algorithms for new properties using (existing) distance-approximation algorithms for other properties. 

	
\begin{theorem}\label{thm:cover}
Let $\calP$ be a graph property. 
Suppose that for every $\eps > 0$, there exists a family $\calF(\eps)$, which is an $\eps$-cover for $\calP$ and
for which the following holds.
For each $\calP' \in \calF(\eps)$ 
there exists an 
 algorithm $\calA_{\calP',\eps}$
that, given $\delta > 0$, queries the induced subgraph over at most $q_\eps\cdot \log(1/\delta)$ random vertices of the input graph $G$, 
and outputs a value $\hDelta_{\calP'}$ such that  $|\hDelta_{\calP'} - \Delta(G,\calP')|\leq \eps/2$
with probability at least $1-\delta$.
		Then there exists a distance approximation algorithm $\calA$ 
		for $\calP$
		whose query complexity is
		$O(\log^2(|\calF(\eps)|)\cdot q_\eps^2)$.
	\end{theorem}
Note that the algorithms referred to in Theorem~\ref{thm:cover} (denoted $\calA_{\calP',\eps}$) are not
precisely distance-approximation algorithms as defined in Definition~\ref{def:dist-approx-alg-add}. This is the case since they are not required to work for any given approximation parameter, 
but rather the quality of their approximation of $\Delta(G,\calP')$ for $\calP' \in \calF(\eps)$ is linked to the quality of the cover $\calF(\eps)$.

\def\ProofOfCoverTheorem{
	Given query access to a graph $G$, the distance approximation algorithm $\calA$ starts by estimating the distance of $G$ to each property $\calP'$ in $\calF =\calF(\eps)$.
	It does so by selecting, uniformly, independently at random
	$s=\Theta\left(q_\eps\cdot \log(3|\calF|)\right)$
	vertices,
	and querying all vertex pairs in the sample to obtain the induced subgraph $H$.
It then runs the distance approximation algorithms $\left\{\calA_{\calP',\eps}\right\}$
	for each $\calP' \in \calF$ with 
	confidence parameter $\delta = 1/(3|\calF|)$ using the (same) sampled subgraph $H$.
	Let $\hDelta_{\calP'}$ be the estimate the algorithm $\calA$ obtains for the distance of $G$ to property $\calP'$.
	It outputs $\hDelta = \min_{\calP'\in \calF} \left\{ \hDelta_{\calP'} \right\}$.
	The query complexity of $\calA$ is $s\choose 2$, 
	as stated in the theorem.
	
	It remains to show that  
$|\hDelta - \Delta(G,\calP)| \leq \eps$
with probability at least	$2/3$.
	By taking a union bound over all executions of 
	the 
algorithms $\calA_{\calP',\eps}$ for the different 
properties $\calP' \in \calF$,
we get that with probability at least 
	$2/3$, for every $\calP'\in \calF$ we have that $|\hDelta_{\calP'} - \Delta(G,\calP') | \leq \eps/2$.
	From this point on, we condition on this event.
	Also, recall that by the premise of this theorem, the family $\calF$ is an $\eps$-cover of $\calP$
	(as defined in Definition~\ref{def:cover}).
	
	Let $G^* \in \calP$ satisfy $\Delta(G,G^*) = \Delta(G,\calP)$ (i.e., $G^*$ is a graph closest to $G$ in $\calP$).
	Let $\calP^* = \argmin_{\calP' \in \calF}\{\Delta(G^*,\calP')\}$.
	By Condition~\ref{it:G-in-P} in Definition~\ref{def:cover},
	$\Delta(G^*,\calP^*) \leq \eps/2$. By applying the triangle inequality, we get that
	$\Delta(G,\calP^*) \leq \Delta(G,\calP) + \eps/2$. It follows that
	\[
	\hDelta \leq \hDelta(\calP^*) \leq  \Delta(G,\calP^*) + \eps/2 \leq \Delta(G,\calP) + \eps \;.
	\]
	For the lower bound on $\hDelta$, let $\widetilde{\calP} = \argmin_{\calP' \in \calF} \{\Delta(G,\calP')\}$,
	and let $\wtG$ be a graph in $\widetilde{\calP}$ for which $\Delta(G,\wtG) = \Delta(G,\widetilde{\calP})$
	(i.e., $\wtG$ is a graph closest to $G$ in $\widetilde{\calP}$).
	By Condition~\ref{it:Gprime-in-union} in Definition~\ref{def:cover}, $\Delta(\wtG,P) \leq \eps/2$, so that
	by the triangle inequality,
	$\Delta(G,\widetilde{\calP}) \geq \Delta(G,\calP) - \eps/2$.
	By the definition of $\widetilde{\calP}$, for every $\calP' \in \calF$, we have that
	$\Delta(G,\calP') \geq \Delta(G,\calP) - \eps/2$. Since we are conditioning on the event that
	$\hDelta(\calP') = \Delta(G,\calP') \pm \eps/2$, we get that
	\[
	\hDelta = \min_{\calP' \in \calF} \left\{\hDelta(\calP')\right\} \geq \min_{\calP' \in \calF} \left\{\Delta(G,\calP')\right\} - \eps/2     \geq \Delta(G,\calP) - \eps \;,
	\]
	as required.
}
\ifnum\fullver=1
\begin{proof}
\ProofOfCoverTheorem
\end{proof}

\else
The proof of Theorem~\ref{thm:cover} appears in Appendix~\ref{app:cover}.
\fi
	The following is a direct corollary of Theorem~\ref{thm:cover} and Lemma~\ref{lem:shpp-are-ee} (proven in Appendix~\ref{app:shpp-csp}).
	\begin{corollary}\label{coro:cover-of-shpp}
		Let $\calP$ be a graph property and $f:[0,1] \to \mathbb{N}$ a function.
Suppose that for every $\eps > 0$, there exists
a family of properties 
$\calF \subset \tcalHPP^{k}$, where  $k = f(\eps)$, such that $\calF$ is an $\eps$-cover for $\calP$.
		Then there exists a distance approximation algorithm $\calA$ 
for $\calP$ whose query complexity is
  $ \poly(1/\eps,\log k ,\log(|\calF|))$.
	\end{corollary}

\section{Induced $P_3$-freeness}\label{sec:p3}

In this section we prove  Theorem~\ref{thm:p3_is_EE} (which as noted in the introduction, is simple, and can be viewed as a warmup for the other results).
For the sake of succinctness, in what follows we refer to induced $P_3$-freeness
as $P_3$-freeness.

In order to prove Theorem~\ref{thm:p3_is_EE}
we establish the next lemma (which will then allow us to apply 
Corollary~\ref{coro:cover-of-shpp}.
\begin{lemma}\label{lem:p3-cover}
For every $\eps>0$, there exists a family of
 graph properties $\calF(\eps) \subset \calHPP^k$, for 
 $k = O(1/\eps)$,
 such that
   $\calF(\eps)$ is an $\eps$-cover for $P_3$-freeness and 
   $|\calF(\eps)|=1$.
\end{lemma}

We use the following known 
characterization of $P_3$-free graphs.
\begin{fact}\label{fact:p3-cliques}
A graph  is $P_3$-free if and only if it is a disjoint union of cliques, i.e., each of its connected components is a clique.
\end{fact}
	
\begin{proof}[Proof of Lemma~\ref{lem:p3-cover}]
Let $k = 1/\eps +1$. We define a partition function $\phi_{k} : [k]\times [k]$ as follows.
For each $i \in [k-1]$, $\phi_{k}(i,i) = 1$, and $\phi_{k}(k,k) = 0$. For each $i,j \in [k]$ such that $i \neq j$,
$\phi_{k}(i,i) = 0$. Thus $\calP_{\phi_{k}}$ consists of all graphs that are a union of at most $k-1$ cliques and an independent set (equivalently, all connected components are cliques, and at most $k-1$ of these cliques contain more than a single vertex). We let $\calF(\eps) = \{\calP_{\phi_{k}}\}$ and show next that $\calF(\eps)$ is an $\eps$-cover for  $P_3$-freeness.

Starting with Condition~\ref{it:G-in-P} in Definition~\ref{def:cover}, consider any graph $G$ that is $P_3$-free.
By Fact~\ref{fact:p3-cliques}, it is a disjoint union of cliques. Let the vertex sets of these cliques be $V_1,\dots,V_t$, where $|V_1|\geq |V_2|\geq \dots \geq |V_t|$. Consider the graph $G'$ that results from $G$ by removing all edges internal to the sets $V_k,\dots,V_t$ (thus turning $\bigcup_{j=k}^t V_j$ into an independent set). This graph belongs to $\calP_{\phi_{k}}$ and
is $(\eps/2)$-close to $G$ (since $|V_j| < \eps n$ for each $j \geq k$, so that for each vertex, the number of edges incident to it that are removed is less than $\eps n$  and each edge is counted twice).

Turning to  Condition~\ref{it:Gprime-in-union} in Definition~\ref{def:cover}, let $G'$ be a graph that belongs to  $\calP_{\phi_{k}}$. By the definition of $\calP_{\phi_{k}}$, all connected components of $G'$ are cliques (some of size $1$). By Fact~\ref{fact:p3-cliques}, $G'$ is $P_3$-free.
\end{proof}

Theorem~\ref{thm:p3_is_EE} now directly follows by combining 
Corollary~\ref{coro:cover-of-shpp} and Lemma~\ref{lem:p3-cover}.


\section{Induced $P_4$-freeness}\label{sec:p4}

In this section we prove  Theorem~\ref{thm:p4_is_EE} (with some details deferred to Appendix~\ref{app:p4}).
For the sake of succinctness, in what follows we refer to induced $P_4$-freeness
as $P_4$-freeness.

In order to prove Theorem~\ref{thm:p4_is_EE}
we establish the next lemma (which will then allow us to apply 
Corollary~\ref{coro:cover-of-shpp}.
\begin{lemma}\label{lem:p4-cover}
For every $\eps>0$, there exists a family of
 graph properties $\calF(\eps) \subset \calHPP^k$, for 
 $k = O(1/\eps)$,
 such that
   $\calF(\eps)$ is an $\eps$-cover for $P_4$-freeness and $|\calF(\eps)| =2^{O(1/\eps^2)}$.
\end{lemma}

In order to define the family $\calF(\eps)$ referred to in Lemma~\ref{lem:p4-cover}, we make use of
a known characterization of $P_4$-free graphs that is stated next, where for 
a node  $y$ in a tree $T$, we let $L_T(y)$ denote the set of leaves in the subtree rooted at $y$.

	\begin{lemma} 
              \label{lem:p4-to-tree}
		Let $G$ be a graph with vertex set $V$. The following statements are equivalent:
		\begin{enumerate}
			\item $G$ has no induced subgraph isomorphic to $P_4$.
			\item \label{it:T-G}
             There exists a binary tree, denoted $T(G)$, for which the following holds.
             \begin{enumerate}
               \item There is a one-to-one correspondence between the leaves of $T(G)$ and the vertices of $G$.\footnote{For the sake of simplicity, we think of the leaves of $T(G)$ as actually being vertices of $G$ (as can be seen in the next item).}
               \item Each internal node $x$ in $T(G)$ corresponds to a homogeneous cut $(U_1,U_2)$,
               where the two children of $x$, denoted $x_1$ and $x_2$, respectively, satisfy $L_{T(G)}(x_1) = U_1$ and $L_{T(G)}(x_2)=U_2$.
             \end{enumerate}
		\end{enumerate}
	\end{lemma}
Lemma~\ref{lem:p4-to-tree} was implicitly proved in~\cite{CLS81}
and we provide the proof in Appendix~\ref{app:p4}  for the sake of completeness.
We note that for a graph $G$, there may be more than one corresponding tree $T(G)$, but $T(G)$ uniquely determines $G$.

Using Lemma~\ref{lem:p4-to-tree} 
we can establish the next lemma (which will be used in the proof of Lemma~\ref{lem:p4-cover}).
\begin{lemma}\label{lem:close-tree}
Let $G = (V,E)$ be a graph that is $P_4$-free.
For any given $\eps >0$, there exists a graph $G' = (V,E')$ that is $(\eps/2)$-close to $G$ for which the following holds.
There is a tree $T(G')$ as defined in Item~\ref{it:T-G} of Lemma~\ref{lem:p4-to-tree}, which includes a set $C$ of $O(1/\eps)$ internal nodes in $T(G')$ such that:
\begin{enumerate}
\item  Every leaf in $T(G')$ (vertex in $V$) belongs to $L_{T(G')}(x)$ for some $x \in C$.
\item All nodes in $C$ and internal nodes descending from them correspond to empty cuts.
\end{enumerate}
\end{lemma}

\def\ProofOfLemCloseTree{
In order to prove Lemma~\ref{lem:close-tree}, we first establish
the following simple combinatorial claim.
\begin{claim} \label{clm:leaf_count_of_subtree}
	For every integer $\ell > 0$, every binary tree $T$ with $n\geq \ell$ leaves has an internal node $y$ such that $\frac{\ell}{2} \leq  |L_T(y)| \leq \ell$. Furthermore, $y$ has a sibling (its parent in the tree has an additional child).
\end{claim}
	\begin{proof} 
Since $T$ is fixed in the proof, we use the shorthand $L(y)$ instead of $L_T(y)$.
		We traverse the tree, starting from the root, in the following way:
		\begin{enumerate}
			\item if we are at node $y$, such that $|L(y)| > \ell$, and $y$ has a single child $y'$, we continue to $y'$. Note that $|L(y')| = |L(y)| \geq \ell$.
			\item if we are at node $y$, such that $|L(y)| > \ell$, and $v$ has two children, at least one of them has at least $|L(y)|/2 \geq \ell/2$ leaves. We continue to it.
			\item if we are at node $y$, such that $|L(y)| \leq \ell$, we return $y$.
		\end{enumerate}
		It is easy to see that we only stop at nodes $y$ s.t $\ell/2 \leq |L(y)| \leq \ell$.
		To see that we eventually stop, note that we start from the root (which has $n > \frac{\ell}{2}$ leaves), and in each step we increase our distance from the root. Therefore, we necessarily eventually stop.
		To see that we stop at a node with a sibling, observe that we stop at the first node we reach for which $L(y) \leq \ell$, and only when moving to a node with a sibling we decrease $L(y)$.
	\end{proof}

\medskip

\begin{proof}[Proof of Lemma~\ref{lem:close-tree}]
We first apply Lemma~\ref{lem:p4-to-tree}, and let $T(G)$ be a tree as defined in Item~\ref{it:T-G} of the theorem.
We next show how to modify  $T(G)$ so as to obtain a tree $T(G')$ for a graph $G'$ that is $(\eps/2)$-close to $G$.
This is done in two phases. In the first phase we do not modify the structure of $T(G)$ but only modify the type of cuts corresponding to some of the internal nodes. In the second phase we also modify the structure of the tree (over the same set of leaves).

\paragraph{Phase I.}
In the first phase we apply Claim~\ref{clm:leaf_count_of_subtree} repeatedly with $\ell = \frac{\eps}{4} n$, to find nodes in $T(G)$ that have between $\frac{\eps}{8} n$ and $\frac{\eps}{4} n$ descending leaves.
Specifically, starting with $j=1$ and  $T^1 = T(G)$, 
in  iteration $j$ we search for a node $x_j$ such that $x_j$ is not an ancestor of any previously selected $x_i$ ($i<j$),  and $\frac{\eps}{8} n \leq |L_{T(G)}(x_j)| \leq \frac{\eps}{4} n$.
If such a node $x_j$ exists, then 
we add $x_j$ to the set of selected internal nodes $C$ and modify $T^j$ as follows. For $x_j$ and all its descending internal nodes, the cuts associated with them are made empty.
This implies that in the graph determined by $T^j$ there are no edges between vertices with both endpoints in $X_j = L_{T(G)}(x_j)$.
%

If we reached an iteration $j$ such that every leaf in $T(G)$ (similarly, $T^j$) is a descendant of some selected $x_i$, $i<j$, then we
set $T(G') = T^j$ and terminate. Observe that in this case
 $|C|  \leq 8/\eps$ and $G'$ differs from $G$ only on pairs of vertices that both belong to the same $X_j$, where there are at most $\frac{\eps}{4}n^2$ such edges.
Otherwise we begin the second phase, as explained next.

\paragraph{Phase II.}
For each node $x_i$, $i =1,\dots,j-1$, let $y_i$ be its parent in $T(G)$ (which is also its parent in $T^j$).
  Let $\{y_j,\dots,y_{j'}\}$ be the set of least common ancestors for all pairs of vertices among $\{y_1,\dots,y_{j-1}\}$
  (so that in particular, this includes the root of $T(G)$).
Let $Y = \{y_1,\dots,y_{j'}\}$ and observe that $|Y| \leq j' \leq 2j \leq 16/\eps$.
Consider the set $S$ of  edges-disjoint paths between pairs of nodes in $Y$ such that no node in $Y$ is internal to any of these paths and each path contains at least three nodes.

  Let $(z_1,\dots,z_t)$ be a path in $S$, so that $t\geq 3$ and $\{z_1,\dots,z_t\} \cap Y = \{z_1,z_t\}$.  The order of the nodes on the path is from minimum to maximum distance to the root.
   Recall that in $T(G)$, every internal node has two children. Therefore, for each $2 \leq r \leq t-1$, the node $z_r$ has a child $w_r \neq z_{r+1}$.
Furthermore, $|L_{T(G)}(w_r)| \leq \frac{\eps}{8} n$ for every  $2 \leq r \leq t-1$,  (or else $w_r$ (or one of its descendants) would have been selected as the next $x_j$).

Assume first that $\sum_{r=2}^{t-1} |L_{T(G)}(w_r)| \leq  \frac{\eps}{4} n$.
Let $p = (z_2,\dots,z_{t-1})$ and $Z(p) = \{z_2,\dots,z_{t-1}\}$. Let $Z_0(p)$ be the subset of nodes in $Z(p)$ that are associated with an empty cut in $T(G)$ (the same holds for $T^j$) and let $Z_1(p)$ be the subset associated with a complete cut. Let $W_0(p)$ be the children of nodes in $Z_0(p)$ that do not belong to $Z(p)$, and let $W_1(p)$ be defined analogously for $Z_1(p)$. We now modify $T^j$ as follows. We replace the path $(z_2,\dots,z_{t-1})$ with a single edge between two nodes, denoted 
$z^{0}(p)$ and $z^{1}(p)$, respectively.
We let $z^{0}(p)$ be the parent of $z^{1}(p)$ (and we have an edge between $z_1$ and $z^{0}(p)$ and an edge between $z^{1}(p)$ and $z_t$). The node $z^{0}(p)$ is associated with an empty cut, and the node $z^{1}(p)$ with a complete cut. We next add a child, $w^{0}(p)$ to $z^{0}(p)$ and a child $w^{1}(p)$ to to $z^{1}(p)$ where $w^{0}(p)$ is the root of a subtree whose leaves are $\bigcup_{w\in W_0(p)} L_{T(G)}(w)$ and $w^{1}(p)$ is the root of a subtree whose leaves are $\bigcup_{w\in W_1(p)} L_{T(G)}(w)$. All internal nodes in these two subtrees are associated with empty cuts, and we add $w^{0}(p)$ and $w^{1}(p)$
to the set of selected internal nodes $C$.  (If either $Z_0$ or $Z_1$ is empty, then we have only one subtree. If $|W_0(p)|=1$, then $w^0(p)$ is a leaf, corresponding to the single vertex in $W_0(p)$ (so that it is not added to $C$), and a similar statement holds for $W_1(p)$.)

If $\sum_{r=2}^{t-1} |L_{T(G)}(w_r)| >  \frac{\eps}{4} n$, then we do the following. We first partition the path $(z_2,\dots,z_{t-1})$ into a minimal number of smaller node-disjoint sub-paths such that for each smaller  sub-path the number of leaves descending from nodes on the sub-path is at most $\frac{\eps}{4} n$. We then apply to each sub-path $p'$ the same ``contraction'' operation as defined above for $(z_2,\dots,z_{t-1})$ (and add $w^0(p')$ and $w^1(p')$ to $C$ for each sub-path $p'$).

Consider applying the above to all paths in $S$ and let $\wtT$ be the resulting tree.
By the construction of $\wtT$, the graph $G'$ for which $\wtT = T(G')$ differs from $G$ only on pairs of vertices that both belong to a common subset $W(p') = W_0(p') \cup W_1(p')$ as defined above (here $p'$ may be a path $p$ as defined above, or a sub-path of $p$).
Since the total number of leaves that descend from the nodes in each $W(p')$ is at most $\frac{\eps}{4} n$, the distance between $G'$ and $G$ is as required.
On the other hand, since for every node $x_j$ selected in the first phase, $|L_{T(G)}(x_j)| = |L_{\wtT}(x_j)| \geq \frac{\eps}{8}n$, and for at least a quarter of the nodes $w^b(p')$ (for $b \in \{0,1\}$) we have that
$|L_{\wtT}(x_j)| \geq \frac{\eps}{4}n$, we get that
 $|C|\leq 16/\eps$.
\end{proof}

} 
\ifnum\fullver=0
The proof of Lemma~\ref{lem:close-tree} is somewhat technical, and is deferred to Appendix~\ref{app:p4}.
\else

\ProofOfLemCloseTree
\fi

\medskip
We are now ready to prove Lemma~\ref{lem:p4-cover}, where we make use of the following notation.
For any two 
nodes $y_1$ and $y_2$ in a tree $T$, their lowest common ancestor is denoted by $\anc_T(y_1, y_2)$.

\begin{proof}[Proof of Lemma~\ref{lem:p4-cover}]
In order to define $\calF(\eps)$, we define a set of homogeneous partition functions $\Phieps$ and let $\calF(\eps) = \{\calP_\phi: \phi \in \Phieps\}$.
Let $k = 32/\eps$. Each $\phi \in \Phieps$ is determined by: (1) a binary tree $T = (V_T,E_T)$ with 
$k' \leq k$ leaves, whose set is denoted by $L_T$, such that each internal node in $T$ has two children;
 (2) a Boolean function $\beta$ over $V_T$ such that $\beta(w)=0$ for every leaf $w \in L_T$; (3)
 a function $\iota : L_T \to [k']$.
The function $\iota$ simply defines a labeling of the leaves of $T$, and the function $\beta$ will be used to determine
cut-types.

The function $\phi_{T,\beta,\iota}$ is defined as follows.
For every $i , i' \in [k']$ we let
$\phi_{T,\beta,\iota}(i,i')=\beta(\anc_T(\iota^{-1}(i),\iota^{-1}(i')))$.
Observe that $\phi_{T,\beta,\iota}(i,i)=0$ for every $i \in [k']$ (so that all parts correspond to independent sets),
while for $i\neq i'$,
the value of $\beta$ on the lowest common ancestor of the leaves mapped by $\iota$  to $i$ and $i'$, respectively,
determines whether the cut between part $i$ and part $i'$ is empty or complete.

As stated above, we let $\calF(\eps)$ contain all graph properties in $\calHPP^k$ defined by 
homogeneous partition functions $\phi_{T,\beta,\iota}$ in $\Phieps$. Since $|\Phieps| \leq 2^{k^2}$, we get the same upper bound on $|\calF(\eps)|$.
It remains to show that $\calF(\eps)$ is an $\eps$-cover of $P_4$-freeness (as defined in Definition~\ref{def:cover}).

To establish that Condition~\ref{it:G-in-P} in Definition~\ref{def:cover} holds, consider any graph $G$ that is $P_4$-free.
We show that there exists $\phi \in \Phieps$ such that $G$ is $(\eps/2)$-close to $\calP_\phi$.
To this end we apply Lemma~\ref{lem:close-tree}. Based on $T(G')$ we define a homogeneous partition function $\phi_{T,\beta,\iota} \in \Phieps$ in a straightforward manner.
Let $T$ be the tree resulting from $T(G')$ by replacing each internal node $x$ in $C$ with a leaf. The function $\iota$ may be an arbitrary function from the leaves of $T$ to $[|C|]$. The function $\beta$ assigns value $0$ to all leaves, and to each internal node it assigns value $0$ or $1$ depending on the type of cut associated with the node in $T(G')$. By this definition, $G'$ obeys $\phi_{T,\beta,\iota}$, as required.

We now turn to  Condition~\ref{it:Gprime-in-union} in Definition~\ref{def:cover}. Consider any $\phi \in \Phieps$, and let $T(\phi)$, $\beta(\phi)$ and $\iota(\phi)$ be such that $\phi = \phi_{T(\phi),\beta(\phi),\iota(\phi)}$.
Let $G' = (V,E')$ be a graph in $\calP_{\phi}$. We show that $G'$ is $P_4$-free by applying Lemma~\ref{lem:p4-to-tree}.
Namely, we show that Item~\ref{it:T-G} in the lemma holds. Let $(V_1,\dots,V_k)$ be a witness partition of $V$ with respect to $\phi$. 
This implies that we can define a tree $T(G')$ as described in
Item~\ref{it:T-G} of Lemma~\ref{lem:p4-to-tree} by essentially ``extending'' $T(\phi)$. To be precise, for each internal node
$x\in T(\phi)$ we have an internal node (also denoted $x$) in $T(G')$. Let $y$ and $z$ be the  children of $x$ in $T(\phi)$. Then in $T(G')$ the cut corresponding to $x$ is between
$\bigcup_{w\in L_{T(\phi)}(y)} V_{\iota(w)}$ and
$\bigcup_{w\in L_{T(\phi)}(z)} V_{\iota(w)}$
(which is either empty or complete, by the definition of $\phi$). For each leaf $x$ in $T(\phi)$,
$V_{\iota(x)}$ is an independent set. 
In $T_{G'}$ we  replace the leaf $x$ with some subtree whose leaves correspond to vertices in $V_{\iota(x)}$. Each internal node in this subtree corresponds to an empty cut between the vertices associated with the leaves descending from its left child and the leaves descending from its right child.
We have thus obtained a tree $T(G')$ as defined in Item~\ref{it:T-G} of Lemma~\ref{lem:p4-to-tree}, implying that $G'$ is $P_4$-free, and hence Item~\ref{it:Gprime-in-union} in Definition~\ref{def:cover} holds, as required.
	\end{proof}

Theorem~\ref{thm:p4_is_EE} now directly follows by combining 
Corollary~\ref{coro:cover-of-shpp}
and Lemma~\ref{lem:p4-cover}.

\section{Induced $C_4$-freeness}\label{sec:c4}

In this section we prove Theorem~\ref{thm:c4_is_EE} (with some details deferred to Appendix~\ref{app:c4}).
In order to prove Theorem~\ref{thm:c4_is_EE}
we establish the next lemma.
\begin{lemma}\label{lem:c4-cover}
For every $\eps>0$, there exists a family of
 semi-homogeneous partition properties $\calF(\eps) \subset \tcalHPP^k$, for 
 $k = \exp(\exp(\poly(1/\eps)))$,
 such that
   $\calF(\eps)$ is an $\eps$-cover for induced $C_4$-freeness and
   $|\calF(\eps)| = \exp(\exp(\poly(1/\eps)))$.
\end{lemma}

Here too, for the sake of succinctness,  we refer to induced $C_4$-freeness
as $C_4$-freeness.
In order to prove Lemma~\ref{lem:c4-cover},
 we make use of the next lemma, which is
essentially implicit in~\cite{GS19} and whose proof can be found in Appendix~\ref{app:c4}.


\begin{lemma}\label{lem:everything_for_assaf_and_lior}
	Let $G = (V,E)$ be a $C_4$-free graph. Then for every $\eps>0$
there exists a graph $G' = (V,E') $ such that the following holds.
	\begin{enumerate}
		\item\label{prop:Delta-G-G'} $\Delta(G,G') \leq \eps/2$.
		\item\label{prop:G'-C4-free} $G'$ is $C_4$-free.
		\item\label{prop:G'-Q-I}  There exists a partition of $V$ into two subsets, $Q$ and $I$ such that
$I$ is an independent set in $G'$
and $Q$ can be further partitioned into subsets $Q_1, \dots, Q_{t}$ for
$t \leq \bar{t} = \exp(\poly(1/\eps))$  
such that the following holds: 
        \begin{enumerate}
            \item\label{prop:Qi-clique} For each $i \in [t]$, $G'[Q_i]$ is a clique.
            \item\label{prop:Qi-Qj-homogeneous} For each $i,j \in [t]$, $i\neq j$,  $(Q_i, Q_j)$ is homogeneous in $G'$.
		    \item\label{prop:Ny-clique} For vertex $y \in I$, $G'[N_{G'}(I)]$ is a clique.
        \end{enumerate}
	\end{enumerate}
\end{lemma}

\begin{proof}[Proof of Lemma~\ref{lem:c4-cover}]
We first define the family $\calF(\eps)$, and then show that it satisfies the requirements stated in the lemma.

\paragraph{Defining {$\calF(\eps)$}.}
In order to define the family of semi-homogeneous partition properties $\calF(\eps)$, we consider the family of graphs $\calH(\eps)$, which consists of all $C_4$-free graphs that contain at most $k_1 = \bar{t} = \exp(\poly(1/\eps))$
nodes.
For each $H= (V(H),E(H))$  in $\calH(\eps)$ we define a semi-homogeneous partition function $\phi_H$. Rather than defining $\phi_H$ over $[k']\times [k']$ for some $k' \leq k$, it will be convenient to define it over $K(H)\times K(H)$, for a set $K(H)$ that satisfies $|K(H)|= k'$. The set $K(H)$ is a union of two (disjoint) subsets: $K_1(H)$ and $K_0(H)$ where
$K_1(H) = V(H)$ and for each subset $C$ in $V(H)$ such that $H[C]$ is a clique, we have an element (part) $u_C \in K_0(H)$.
The function $\phi_H$ is defined as follows.

\begin{enumerate}
\item For every $v\in K_1(H)$, we set $\phi_H(v,v) = 1$.
\item For every $v,v' \in K_1(H)$ such that $v \neq v'$ we set $\phi_H(v,v')=1$ if $(v,v') \in E(H)$, and otherwise
$\phi_H(v,v')=0$.
\item For every two (not necessarily distinct) $u_C,u_{C'} \in K_0(H)$ we set $\phi_H(u_C,u_{C'}) = 0$.
\item For every $u_C \in K_0(H)$, and for every $v \in C$, we set $\phi_H(u_C,v) = \bot$, and for every $v\in K_1(H)\setminus C$ we set $\phi_H(u_C,v) = 0$.
\end{enumerate}
We now let
\[\calF(\eps) = \{\calP_{\phi_H}: \; H\in \calH(\eps)\}\;.\]

\paragraph{Establishing Condition~\ref{it:G-in-P} in Definition~\ref{def:cover}.}
Let $G$ be a $C_4$-free graph. We apply Lemma~\ref{lem:everything_for_assaf_and_lior} to obtain a graph $G'$ with the properties stated in the lemma (so that in particular, by, Property~\ref{prop:Delta-G-G'} in the lemma, $\Delta(G,G') \leq \eps/2$). We next show that there exists a graph $H \in \calH(\eps)$ such that $G'$ obeys $\phi_{H}$. For each $Q_i$ in the partition of $Q$ (ensured by Property~\ref{prop:G'-Q-I}), we have a node $v_i$ in $V(H)$. There is an edge between $v_i$ and $v_j$ in $V(H)$ if and only if the cut $(Q_i,Q_j)$ is complete in $G'$ (recall that $(Q_i,Q_j)$ is homogeneous by Property~\ref{prop:Qi-Qj-homogeneous}). We need to show that $H \in \calH(\eps)$ and that $G'$ obeys $\phi_{H}$.

In order to show that $H \in \calH(\eps)$, we first note that since there is a vertex in $V(H)$ for each subset $Q_i$, we have that $|V(H)| = t \leq \bar{t}$ 
as required. We next verify that $H$ is $C_4$-free. Assume, contrary to the claim, that $H$ contains four nodes, $v_{i_1},v_{i_2},v_{i_3},v_{i_4}$ such that the subgraph of $H$ induced by these nodes is a $C_4$.
But then this implies that there exists an induced $C_4$ in $G'$ (take $q_j \in Q_{i_j}$ for $j\in [4]$), contradicting the fact that $G'$ is $C_4$-free.

It remains to show that  $G'$ obeys $\phi_{H}$. To this end we  assign the vertices in $V(H)$ to parts of $\phi_H$.
We start by assigning each $q\in Q_i\subseteq Q$ for $i \in [t]$ to $v_i$.
By the properties of the partition $(Q_1,\dots,Q_t)$ of $Q$ (Properties~\ref{prop:Qi-clique} and~\ref{prop:Qi-Qj-homogeneous} in
Lemma~\ref{lem:everything_for_assaf_and_lior}) and the definition of $\phi_H$,
\begin{equation}\label{eq:q-q'}
\forall q,q'\in Q,\;q\neq q',\; q\in Q_i,q'\in Q_j:\; (q,q')\in E(G')\; \mbox{ i.f.f. } \; \phi_H(v_i,v_j) = 1\;.
\end{equation}
Next,
for each $y \in I$ (where by Property~\ref{prop:G'-Q-I} in Lemma~\ref{lem:everything_for_assaf_and_lior}, $G'[I]$ is an 
empty graph),
let $J(y) = \{v_i : \; N_{G'}(y) \cap Q_i \neq \emptyset\}$.
Recall that by Property~\ref{prop:Ny-clique}, the subgraph of $G'$ induced by $N_{G'}(y)$ is a clique for every $y \in I$.
Therefore, for every $y \in I$, the subgraph of $H$ induced by $J(y)$ is a clique as well.
Also recall that by the definition of $\phi_H$, each part $u_C \in K_0(H)$ is indexed by a clique $C$ in $H$.
Hence, we can assign each $y \in I$ to $u_{J(y)}$.
By the definition of $\phi_H$,
\begin{equation}\label{eq:y-y'}
\forall y,y' \in I,\; \phi_H(u_{J(y)},u_{J(y')})=0\;,
\end{equation}
which is consistent with the fact that $G'[I]$ is an 
empty graph.
We also have that
\begin{equation}\label{eq:y-q}
\forall y\in I,\;  \phi_H(u_{J(y)},v_i) = \bot \; \mbox{ if }\; i \in J(y), \;\mbox{ and }
\phi_H(u_{J(y)},v_i) = 0\; \mbox{ if }\;  i \notin J(y)\;.
\end{equation}
By combining Equations~(\ref{eq:q-q'})--(\ref{eq:y-q}), we get  that $G'$ obeys $\phi_H$.

\paragraph{Establishing Condition~\ref{it:Gprime-in-union} in Definition~\ref{def:cover}.}
Consider a graph $G' \in \calP$ for some $\calP \in \calF(\eps)$. We next show that it is $C_4$-free. By the definition of $\calF(\eps)$, the property $\calP$ is defined by some semi-homogeneous partition function $\phi_H$ for
$H \in \calH(\eps)$. Namely, there exists a mapping $\nu: V(G') \to K(H) = K_1(H) \cup K_0(H)$ such that for every
$z,z' \in V(G')$, if $\phi_H(\nu(z),\nu(z') = 1$ then $(z,z')\in E(G')$ and if $\phi_H(\nu(z),\nu(z') = 0$ then $(z,z')\notin E(G')$.

Assume, contrary to the claim, that $G'$ contains an induced $C_4$ over the (distinct) vertices
$z_1,z_2,z_3,z_4 \in V(G')$ (with the edges $(z_1,z_2)$, $(z_2,z_3)$, $(z_3,z_4)$ and $(z_4,z_1)$),
and let  $Z=\{z_j\}_{j=1}^4$.
%
Consider the parts $\nu(z_1),\nu(z_2),\nu(z_3),\nu(z_4)$ (which are not necessarily distinct).
\begin{enumerate}
\item Suppose that at least three among these parts belong to $K_0(H)$.
But $\phi_H(u,u')=0$ for every pair $u,u' \in K_0(H)$, so this case is not possible (under the counter assumption that $G'[Z]$ is a $C_4$).
\item Suppose that all four of these parts belong to $K_1(H)=V(H)$. These parts cannot be distinct, as $H$ is $C_4$-free, and no two can be identical since then the number of edges in $G'[Z]$ is at least 5.
\item If at least one part belongs to $K_0(H)$ and at least two parts belong to $K_1(H)$, then without loss of generality,
$\nu(z_1) \in K_0(H)$ and $\nu(z_2),\nu(z_4) \in K_1(H)$.
By the definition of $\phi_H$, this means that $\phi_H(\nu(z_1),\nu(z_2))= \bot$ and $\phi_H(\nu(z_1),\nu(z_4))=\bot$. But then $(\nu(z_2),\nu(z_4))\in E(H)$, so that $\phi_H(\nu(z_2),\nu(z_4)) = 1$, implying that $(z_2,z_4) \in E(G')$, so that
$G'[Z]$ is not a $C_4$.
\end{enumerate}

\paragraph{Bounding $k$ and $|\calF(\eps)|$.}
Finally, by the definition of $\calF(\eps)$ we have that for every $H$,
$|K(H)| = |K_1(H)| + |K_0(H)| \leq k_1 + 2^{k_1}$ which is upper bounded by $\exp(\exp(\poly(1/\eps)))$.
As for the upper bound on $|\calF(\eps)|$, we have that
$|\calF(\eps)| = |\calH(\eps)| =  2^{2^{O(k_1^2)}}$ 
which is also upper bounded by $\exp(\exp(\poly(1/\eps)))$,
and the lemma is established.
\end{proof}
Theorem~\ref{thm:c4_is_EE} now directly follows by combining Corollary~\ref{coro:cover-of-shpp} and Lemma~\ref{lem:c4-cover}.


\section{Chordality}\label{sec:chordal}



In this section we prove the next lemma.
	\begin{lemma} \label{lem:chordal_cover}
		For every $\eps>0$, there exists a family of graph properties
		$\calF(\eps) \subset \tcalHPP^k$, for 
$k=\exp(\poly(1/\eps))$
		such that
		$\calF(\eps)$ is an $\eps$-cover for chordality and 
	$|\calF(\eps)|=\exp(\poly(1/\eps))$.
	\end{lemma}
Theorem~\ref{thm:chordaily_is_EE} directly follows by combining Lemma~\ref{lem:chordal_cover} with
Corollary~\ref{coro:cover-of-shpp}.


	In Section~\ref{subsec:chordal-preliminaries} we present some basic definitions and a claim regarding chordal graphs.
In Section~\ref{subsec:chordal-main-lemmas} we state the main lemmas that the proof of Lemma~\ref{lem:chordal_cover} is based on, and show how Lemma~\ref{lem:chordal_cover} is derived from them. 
We prove these lemmas in Section~\ref{subsec:chordal_x_part}, Section~\ref{subsec:P-T-lemmas} and the appendix.

	
	\subsection{Chordal Preliminaries} \label{subsec:chordal-preliminaries}

The notion of \emph{clique trees} was introduced independently by Buneman~\cite{Buneman}, Gavril~\cite{Gavril} and Walter~\cite{Walter} to characterize chordal graphs. Here we use a slight variant, which appears in particular in~\cite[Sec. 3]{Chordal-Intro}. 
\begin{definition} 
\label{def:clique-tree}
		A {\sf clique tree} of a 
 graph $G$ is a tree $T_G$
 such that there is a one-to-one correspondence between the nodes of $T_G$ and the maximal cliques of $G$, and
 where for every two maximal cliques $C,C'$ in $G$, each clique on the path from $C$ to $C'$ in $T_G$ contains $C\cap C'$.
	\end{definition}	
For simplicity of the presentation, when referring to a node in $T_G$ that corresponds to a maximal clique $C$ in $G$, we sometimes simply refer to the node as $C$.

The proof of the following theorem can be found in~\cite[Thm. 3.2]{Chordal-Intro}.
To be precise, the proof is given for connected graphs, but it is not hard to verify that it also holds for general graphs.
For details, see Appendix~\ref{app:subsec:proof_thm+chordal_iff_clque_tree}.
	\begin{theorem} 
\label{thm:chordal-iff-clique-tree}
		A graph $G$ is chordal if and only if it has a clique tree.
	\end{theorem}

The next three definitions will be helpful in analyzing the structure of chordal graphs.
	
	\begin{definition}\label{def:plain}
		A \emph{\sf non-branching}
       path on a tree is a path whose internal nodes are all of degree two.
   \end{definition}

   \begin{definition}\label{def:side}
       Let $T$ be a tree and let $x$ and $y$ be two nodes in $T$ such that the path between them, denoted $P$, is non-branching. We shall say that a node $z$ in $T$ that does not belong to $P$ is {\sf on the same side of $T$} as $x$, if the shortest path on $T$ from $z$ to $y$ passes through $x$. Otherwise it is on the same side of $T$ as $y$.
      	\end{definition}

	\begin{definition}\label{def:clique-cover}
A {\sf clique cover} of a graph $G = (V,E)$ is a set of 
maximal cliques in $G$ whose union is $V$. 
We say that it is an {\sf $r$-clique cover} if it contains $r$ cliques.
A graph that has an $r$-clique cover is said to be \emph{$r$-clique-coverable}.
	\end{definition}

The following lemma is similar to Lemma~\ref{lem:everything_for_assaf_and_lior}, which was used in the analysis of
distance approximation for induced $C_4$-freeness (recall that every chordal graph is in particular induced $C_4$-free).

	\begin{lemma} \label{lem:chordal_everything_for_assaf_and_lior}
		Let $G= (V,E)$ be a chordal graph, and let $\eps' \in (0,1)$. Then there exists a chordal
graph  $G'=(V,E')$ such that: 
		\begin{enumerate}
			\item \label{it:chordal-G-GP-epsP} $\Delta(G,G') \leq \eps'$.

			\item \label{it:chordal-Q-I} $V$ can be partitioned into two subsets $Q$ and $I$ for which the following holds.
          \begin{enumerate}
          \item \label{it:chordal-Q} $G'[Q]$ is chordal and has an $r$-clique cover  for $r \leq 2^{20}/(\eps')^{12}$. 
          \item \label{it:chordal-I}  $G'[I]$ is empty
          and for every vertex $y\in I$, $N_{G'}(y)$ is a clique in $G'$.
        \end{enumerate}
		\end{enumerate}
	\end{lemma}
The proof of Lemma~\ref{lem:chordal_everything_for_assaf_and_lior} is given in Appendix~\ref{app:chordal}.

\subsection{Main supporting lemmas and the proof of Lemma~\ref{lem:chordal_cover}} \label{subsec:chordal-main-lemmas}

The next lemma is a central component in the proof of Lemma~\ref{lem:chordal_cover}.
\begin{lemma} \label{lem:chordal_x_part}
	Let $G$ be an $r$-clique-coverable chordal graph and let $\eps' \in (0,1)$. Then there exists a graph $G'$ such that:
	\begin{enumerate}
		\item $\Delta(G,G') \leq 2r \cdot \eps'$. \label{item:chordal_x_part:dist}
		\item $G'$ is chordal, and  has a clique tree of size  at most $4r/\eps'$.
\label{item:T-G-prime is small}
		\item For each clique $C$ in $G$, there exists a maximal clique $C'$ in $G'$ s.t. $|C\setminus C'| \leq \eps' n$. \label{item:x_part:close_cliques}
	\end{enumerate}
\end{lemma}
The proof of Lemma~\ref{lem:chordal_x_part} can be found in Section~\ref{subsec:chordal_x_part}.

In order to state the next two main supporting lemmas (and in order to define the family $\calF(\eps)$, referred to in Lemma~\ref{lem:chordal_cover}), we introduce the next definition.

  \begin{definition}\label{def:chordal-partition}
	Let $T$ be a tree
and let $\calK(T)$ be the set of connected subgraphs of $T$ (subtrees).
Designating a part for each $K\in \calK(T)$, we
define a homogeneous partition function $\phi_T$ over $\calK(T) \times \calK(T)$ as follows.
For every $K_1,K_2 \in \calK(T)$, we let $\phi_T(K_1,K_2) = 1$ if and only if $K_1$ and $K_2$ intersect.
\end{definition}

In Section~\ref{subsec:P-T-lemmas} we prove the next two lemmas.
\begin{lemma} \label{lem:P_accepts_trees}
	Let $T$ be a tree and let $G$ be a chordal graph  that has a clique tree $T_G$ isomorphic to $T$. Then $G$ satisfies 
$\phi_T$.
	Furthermore, there exists a witness partition of $G$ for $\phi_T$ 
such that the following holds. Consider any vertex $v$ in $G$ and a clique $C$ that it belongs to.
	Then the part that $v$ 
belongs to in the aforementioned witness partition
	corresponds to a subgraph of $T$ such that the isomorphic subgraph in $T_G$ contains $C$.
\end{lemma}

\begin{lemma} \label{lem:strict_prop_accepts_only_chordal}
	For any tree $T$, every graph in  
$\calP_{\phi_T}$ is chordal.
\end{lemma}

We are now ready to prove Lemma~\ref{lem:chordal_cover}.

\begin{proof}[Proof of Lemma~\ref{lem:chordal_cover}]
We first define, for any given tree $T$, a semi-homogeneous partition function $\tilde{\phi}_T$, based on
$\phi_T$. The partition function $\tilde{\phi}_T$ is defined over $(\calK(T) \cup \calK'(T)) \times (\calK(T) \cup \calK'(T))$,
where for each node $x$ in $T$ we have a corresponding part $\pi(x) \in \calK'(T)$.
\begin{itemize}
\item For each pair of parts $K_1,K_2 \in \calK(T)$, we let $\tilde{\phi}_T(K_1,K_2) = \phi(K_1,K_2)$. \item     For each pair of nodes $x_1$ and $x_2$ in $T$, we consider the corresponding parts $\pi(x_1),\pi(x_2)\in \calK'(T)$ let $\tilde{\phi}(\pi(x_1),\pi(x_2))=0$.
\item For each node $x$ in $T$ and part $K \in \calK(T)$, we let $\tilde{\phi}(\pi(x),K)=\bot$ if
$x\in K$, otherwise $\tilde{\phi}(\pi(x),K)=0$.
\end{itemize}
We now set
	$$\calF(\eps) = \{\calP_{\tilde{\phi}_T}: \; T \mbox{ is a tree of size at most }  2^{150}/ \eps^{25} \}\;.$$

	A well known result on Catalan numbers bounds the number of unlabeled trees with $k$ vertices by $2^{2k}$. Therefore $|\calF(\eps)| \leq 2^{2^{31}/\eps^{7}}$.
By the definition of 
$\tilde{\phi}_T$,
the number of parts
over which $\tilde{\phi}_T$ is defined is 
upper bounded by $2^{2^{150}/ \eps^{25} + 1}$, so that
$\calF(\eps) \subset \tcalHPP^k$, for 
		$k = \exp(2^{150}/ \eps^{25})$.
It remains to prove that $\calF(\eps)$ is an $\eps$-cover of the family of chordal graphs.

We start by establishing Condition~\ref{it:G-in-P} in the definition of an $\eps$-cover (Definition~\ref{def:cover}).	
	Let $G= (V,E)$ be a chordal graph. We transform it into a graph that belongs to some $\calP$ in
$\calF(\eps)$ by performing the following steps.
\begin{enumerate}
\item Apply Lemma~\ref{lem:chordal_everything_for_assaf_and_lior} to $G$ with $\eps' = \eps_1 = \eps/8$.
let $G_1= (V,E_1)$ be the graph obtained (denoted $G'$ in the lemma)
  and let $Q \subseteq V$ and $I\subseteq V$ be as defined in the lemma.
  In particular we have that $G_1[Q]$ a chordal graph that is $r$-clique coverable for $r\leq  2^{20}/(\eps_1)^{12} = 2^{60}/(\eps)^{12}$
  and that $\Delta(G,G_1) \leq \eps/8$.

\item  Apply Lemma~\ref{lem:chordal_x_part} to $G_1[Q]$ with $r = 2^{60}/(\eps)^{12}$ and $\eps'= \eps_2=\eps/(32r)$ to obtain
a graph $G'_2 = (Q,E'_2)$.
\item Let $G_2=(V,E_2)$ be the graph obtained from $G_1$ by replacing $G_1[Q]$ with $G'_2$ (all other edges in $G_2$ are as in $G_1$).
By Lemma~\ref{lem:chordal_x_part}, $\Delta(G_1,G_2) \leq 4r \cdot \eps_2 = \eps/8$.
\item Let $G_3$ be the graph obtained from $G_2$ by removing edges in the cut $(I,Q)$ as described next. By Lemma~\ref{lem:chordal_everything_for_assaf_and_lior}, the neighbors of each vertex $v\in I$ constitute a clique $C(v)\subset Q$ in $G_1$.
    By 
    Lemma~\ref{lem:chordal_x_part},
    there exists a maximal clique $C'(v)$ in $G_2[Q]=G'_2$ such that $|C(v)\setminus C'(v)| \leq \eps_2 n$.
    For each $v\in I$ and $u\in C(v)\setminus C'(v)$, we remove the edge $(v, u)$.
     (so that the neighbors of $v$ in $G_3$ constitute a clique).
    Observe that $\Delta(G_2, G_3) = \frac{1}{n^2}\cdot \sum _{v\in I} |C'(v)\setminus C(v)| \leq \eps_2 \leq \eps /8$.
\end{enumerate}
	By Lemma~\ref{lem:chordal_x_part}, $G_2[Q] = G'_2$ has a clique tree $T_{G'_2}$ with $4r/\eps_2 \leq 2^{150}/ \eps^{25}$ nodes.
Let $T_2$ be the tree it is isomorphic to (i.e., ignoring the correspondence in $T_{G'_2}$ between nodes and cliques).
    We now show that $G_3 \in \calP_{\tilde{\phi}_{T_2}}$.
	By Lemma~\ref{lem:P_accepts_trees}, $G_2[Q] \in \calP_{{\phi}_{T_2}}$. Namely, there exists an assignment
of the vertices in $Q$ to $\calK(T_2)$ (the connected subgraphs of $T_2$) that obeys the constraints imposed by
$\phi_{T_2}$. 
As for the vertices in $I$, consider each vertex $v\in I$, and the maximal clique $C'(v)$, as defined above. Let $x(C'(v))$ be the node in $T_2$ that $C'(v)$ corresponds to in $T_{G'_2}$. Then $v$ is assigned to $\pi(x(C'(v)))\in \calK'(T_2)$.

To verify that the assignment obeys $\tilde{\phi}_{T_2}$, consider any pair of vertices $v_1$ and $v_2$ in $G$. If they both belong to $Q$, then, as noted above, the assignment obeys $\phi_{T_2}$ and hence $\tilde{\phi}_{T_2}$. If they both belong to $I$, then they are assigned to parts $K_1$ and $K_2$ that are both in $\calK'(T_2)$, and $\tilde{\phi}_{T_2}(K_1,K_2)=0$ which is consistent with the fact that there is no edge between $v_1$ and $v_2$ in $G_3$ as $I$ is an independent set. Finally, consider the case that $v_1\in I$ and $v_2\in Q$.

Let $K_1$ and $K_2$ be the parts they are assigned to, respectively. Recall that by
	(the second part of) Lemma~\ref{lem:P_accepts_trees}, every vertex in $N_{G_3}(v_1)$ is assigned to a connected subgraph in $T_2$
	such that the isomorphic subgraph in $T_{G'_2}$ contains $C'(v)$. Therefore, by the definition of $\tilde{\phi}_{T_2}$, if $v_2$ is a neighbor of $v_1$, then  $\tilde{\phi}_{T_2}(K_1,K_2)=\bot$, and if $v_2$ is not a neighbor of $v_1$, then either
$\tilde{\phi}_{T_2}(K_1,K_2)=\bot$ or $\tilde{\phi}_{T_2}(K_1,K_2)=0$.

By the above sequence of transformations,
$$\Delta(G, \calF(\eps)) \leq \Delta(G, G_1) + \Delta(G_1, G_2) + \Delta(G_2, G_3) \leq \eps\cdot \frac{3}{8} < \eps/2\;.$$

It remains to show that every graph $G\in \calP$ for  $\calP \in \calF(\eps)$ is chordal (thus
establishing Condition~\ref{it:Gprime-in-union} in Definition~\ref{def:cover}).
Let $T$ be a tree for which $G \in \calP_{\phi_T}$.
	Assume by way of contradiction that there exist $k>3$ vertices, $v_1,\dots,v_k$ such that the subgraph of $G$ induced by these vertices is a cycle in $G$ (where $v_i$ is connected to $v_{i+1}$ for every $i \in [k-1]$ and $v_k$ is connected to $v_1$).
If all $k$ vertices $v_1,\dots,v_k$ belong to parts in $\calK(T)$, then we reach a contradiction by
Lemma~\ref{lem:strict_prop_accepts_only_chordal} (since $\tilde{\phi}_T(K_1,K_2)=\phi_T(K_1,K_2)$ for every pair of parts $K_1,K_2\in \calK(T)$).
 Therefore, at least one of these vertices belongs to a part $K' \in \calK'(T)$, where $K'=\pi(x)$ for some node $x$ in $T$.
 Assume without loss of generality that $v_1 \in K'$. By the definition of $\tilde{\phi}_T$, vertices $v_2$ and $v_k$ must belong to parts $K_1,K_2 \in \calK(T)$ that intersect on $x$.
 But then there is an edge between them, and we reach a contradiction (as $k > 3$).
\end{proof}

\subsection{Proof of Lemma~\ref{lem:chordal_x_part}} \label{subsec:chordal_x_part}

In order to prove Lemma~\ref{lem:chordal_x_part} we first establish several claims regarding chordal graphs.
We start with the following definition, which will play an important role in the proof of the lemma.
\begin{definition}\label{def:simplification}
		Let $G=(V,E)$ be a chordal graph, let $C_1$ and $C_2$ be two maximal cliques in $G$, and let $\eps' \in (0,1)$.
We say that a graph $G'=(V,E')$ is a {\sf $(C_1, C_2, \eps')$-simplification} of $G$ if
$G'$ is  obtained from $G$ by deleting a subset of edges in the cut $(C_1\setminus C_2, C_2\setminus C_1)$ and the following conditions hold.
\begin{enumerate}
			\item $\Delta(G, G') \leq \eps'$.\label{item:G-Gp-close}
\item\label{item:short-path} The
subgraph of $G'$ induced by $C_1 \cup C_2$ is chordal
and  has a clique tree that is a path of length at most $2/\eps'$, with $C_1$ and $C_2$ as its endpoints.
			\item \label{item:close-clique}
For every clique  $C \subseteq C_2\cup C_1$ in $G$, there exists a maximal clique $C'$ in $G'$ such that $|C \setminus C'| \leq \eps' n$.
		\end{enumerate}
\end{definition}

The proof of the next lemma builds on~\cite{GS19}, and is provided in Appendix~\ref{app:chordal}.

	\begin{lemma}\label{lem:two_clique_partition}
Let $G=(V,E)$ be a chordal graph, let $C_1$ and $C_2$ be two maximal cliques in $G$,
and let $\eps' \in (0,1)$.
Then $G$ has a $(C_1, C_2, \eps')$-simplification $G' = (V,E')$.
\end{lemma}


\begin{claim} \label{claim:leaves_are_X}
Let $G$ be a chordal graph, let $T_G$ be a clique tree of $G$, and let $\calC$ 
be a clique cover of $G$.
	Then all leaves in $T_G$ 
belong to $\calC$.
\end{claim}
\begin{proof}
	Assume by way of contradiction that there exists a leaf (clique) $C_1$ in $T_G$ that does not belong to $\calC$. Let $C_2$ be  
the parent of $C_1$ in $T_G$. Since $C_2$ is on the path between $C_1$ and every $\barC \in \calC$, we have that
$C_1 \cap \barC \subseteq C_2$ for every  $\barC \in \calC$.
As $C_1 \subseteq \bigcup_{C\in\calC} C$, we get that 
$C_1 \subseteq C_2$, in contradiction to the maximality of $C$.
\end{proof}


\begin{claim} \label{clm:clique_in_path_is_in_union}
	Let $G$ be a chordal graph with 
  a clique tree $T_G$ and a clique cover $\calC$. Let $C_1, C_2$ be two maximal cliques in $G$, such that
the path between them in $T_G$ is non-branching, and 
no internal node on the path corresponds to a maximal clique in $\calC$.
  Let the set of maximal cliques 
  on this path be denoted by $\calP$.
	Then a maximal clique $\barC$ is in $\calP$ if and only if $\barC \subset C_1 \cup C_2$.
\end{claim}
\begin{proof}
	Let $\barC \in \calP$ be a maximal clique in $G$, let $v$ be a vertex in $\barC$
and let $C(v)$ be the maximal clique in the clique cover that $v$ belongs to. Since the path between $C_1$ and $C_2$ in $T_G$ is non-branching and $C(v)$ does not 
correspond to any internal node on this path,
$C(v)$ is either on the side of $C_1$ in $T_G$ or on the side of $C_2$ (see Definition~\ref{def:side} for the notion of ``side''). If the first case,
as $T_G$ is a clique tree, we get that $v \in C(v)\cap \barC \subseteq  C_1$, and in the latter case we get that $v\in C_2$.
Thus $\barC \subset C_1 \cup C_2$.

	To prove the other direction, let $\barC$ be some maximal clique, such that $\barC \subset C_1 \cup C_2$. Let us assume by way of contradiction that $\barC \notin \calP$.
	Assume without loss of generality that $\barC$ is on the same side of $T_G$ 
as $C_1$. Therefore $C_1$ is on the path between $\barC $ and $C_2$, and thus $\barC \cap C_2 \subset C_1$. As $\barC \subset C_1 \cup C_2$, this implies that $\barC \subset C_1$, contradicting the maximality of $\barC$.
\end{proof}

We build on Claim~\ref{clm:clique_in_path_is_in_union} to prove the next claim.
\begin{claim} \label{clm:tree_is_the_same_after_fix}
	Let $G$ be a chordal graph with a clique graph $T_G$ and an $r$-clique-cover $\calC$. 
Let $C_1, C_2$ be two maximal cliques in $G$, such that the path between them in $T_G$ is non-branching and
no internal node on this path corresponds to a maximal clique in $\calC$.
	Let $G'$ be a $(C_1,C_2,\eps')$-simplification of $G$.
	Then
	\begin{enumerate}
		\item $G'$ is $r$-clique-coverable, with $\calC$ as a clique cover. 
                    \label{item:r_coverability}
		\item $G'$ is chordal and has a clique tree $T_{G'}$, that is the same as $T_G$, except
that the path between $C_1$ and $C_2$ in $T_G$ is replaced by a different path in $T_{G'}$.
           \label{item:T-G-prime}
	\end{enumerate}
\end{claim}
\begin{proof}
We first observe that $C_1$ and $C_2$ are also cliques in $G'$, because $G$ and $G'$ differ only on edges with one endpoint in
$C_1 \setminus C_2$ and one endpoint in  $C_2 \setminus C_1$.
We next show that all cliques in $G$ that are not in the induced subgraph of $C_1 \cup C_2$ are also cliques in $G'$.
Let $\barC \not \subset C_1 \cup C_2$ be a clique in $G$. 	
	Applying Claim~\ref{clm:clique_in_path_is_in_union}, we get that $\barC$ is not on the path between $C_1$ and $C_2$. Assume without loss of generality that $\barC$ is on the side of $C_1$ in $T_G$. Then $\barC \cap C_2 \subset C_1 $ implying that $\barC \cap (C_2 \setminus C_1) = \emptyset$.
Hence there is no vertex pair $(u,v)$ such that $u,v \in  \barC$, $ u\in C_1 \setminus C_2$ and $v \in C_2 \setminus C_1$,
implying that $\barC$ is a clique in $G'$.

Hence, all maximal cliques in $G$ that are not 
subsets of $C_1 \cup C_2$ are also cliques in $G'$.
As $G'$
is a subgraph of $G$, every clique in $G'$ is also a clique in $G$.
Therefore, the set of maximal cliques in $G'$ that are not contained in $C_1\cup C_2$ is the same as the set of maximal cliques in $G$ that are not contained in $C_1\cup C_2$. Let us denote this set by $\calC'$.
By the premise of the lemma regarding the cover $\calC$ 
and by Claim~\ref{clm:clique_in_path_is_in_union},
all cliques in the cover $\calC$ 
are in $\calC'$.
Therefore, $\calC'$ 
is also a clique cover of $G$, so that
$G'$ is $r$-coverable, proving Item~\ref{item:r_coverability}.

Turning to Item~\ref{item:T-G-prime},
since $G'$ is a $(C_1,C_2,\eps')$-simplification of $G$, by Item~\ref{item:short-path} in Definition~\ref{def:simplification},
The subgraph of $G'$ induced by $C_1 \cup C_2$ 
has a clique tree that is a path $\calP$ of length at most $2/\eps'$, with $C_1$ and $C_2$ as its endpoints.
	Let $T_{G'}$ be the tree derived from $T_G$ by replacing the path from $C_1$ to $C_2$ with $\calP$.
	Clearly, $T_{G'}$ contains nodes corresponding to all maximal cliques in $G'$, and does not contain any node corresponding to a set of vertices that is not a maximal clique in $G'$.

In order to prove that $T_{G'}$ is a clique tree of $G'$, we need to show that for any two maximal cliques 
$C_a$ and $C_b$, and any clique $\barC$ on the path between them in $T_{G'}$,
$C_a \cap C_a \subseteq \barC$.
We prove this claim by consider the following cases.

\begin{enumerate}	
\item The path between $C_a$ and $C_b$ does not intersect $\calP$. In this case the claim follows from the fact that $T_G$ is a valid clique tree of $G$.
	
\item
$C_a$ and $C_b$ both do not belong to $\calP$. Therefore, one of them is on the side of $C_1$ in $T_{G'}$ and the other on the side of $C_2$. Assume without loss of generality that $C_a$ is on the side of $C_1$.
We consider two subcases.
   \begin{enumerate}
     \item $\barC \notin \calP$. In this subcase the claim follows from $T_G$ being a valid clique tree, as $\barC$ must have been on the path from $C_a$ to $C_b$ in $T_G$.
	
    \item $\barC \in \calP$.
	In this subcase, since $C_1$ and $C_2$ must both be on the path between $C_a$ and $C_b$ in $T_G$,  $C_a \cap C_b \subset C_1$ and $ C_a \cap C_b \subset C_2$. Therefore, $C_a \cap C_b \subset C_1 \cap C_2$. By
the construction of $T_{G'}$ 
the path between $C_1$ and $C_2$ on $T_{G'}$ is a clique tree of the induced graph of $C_1\cup C_2$ in $G'$.  Therefore for every $\barC\in \calP$ we have that $C_a \cap C_b \subset C_1 \cap C_2 \subset \barC$.
  \end{enumerate}
	
\item $C_a$ and $C_b$ both belong to $\calP$. In this case $\barC$ must also belong to $\calP$. The claim follows from $\calP$ being a valid clique tree for $G'[C_1\cup C_2]$ (by 
    the construction of $T_G'$).
	
\item	 $C_a \in \calP$ and $C_b \notin \calP$.
	Assume without loss of generality that $C_b$ is on the side of $C_1$ in $T_{G'}$ 
	As $\calP$ is a valid clique tree (path) for $G'[C_1\cup C_2]$,
by Claim~\ref{clm:clique_in_path_is_in_union},
$C_a \subset C_1 \cup C_2$. Therefore, it is enough to prove that $C_a \cap C_b \cap C_1 \subset \barC$ and that $C_a \cap C_b\cap C_2 \subset \barC$. As $C_b \cap C_2 \subset C_1$ (by the validity of $T_G$) $C_1\cap C_b\cap C_2\subset C_a \cap C_b\cap C_1$. Hence it suffices to prove that $C_a \cap C_b\cap C_1 \subset \barC$.

	If $\barC \in \calP$, as $\calP$ is a valid clique tree for $G'[C_1\cup C_2]$, $C_a\cap C_1 \subset \barC$, and the claim follows.
	If $\barC \notin \calP$, as $C_1, \barC, C_b$ are not in $\calP$, and $T_G$ is a valid clique tree for $G$, $C_1 \cap C_b \subset \barC$, and the claim follows.
\end{enumerate}	
Item~\ref{item:T-G-prime} is thus established.
\end{proof}

We are now ready to prove Lemma~\ref{lem:chordal_x_part}.

\begin{proof}[Proof of Lemma~\ref{lem:chordal_x_part}]
	Let $T_G$ be a clique tree of $G$, let $\calC = \{\maxX_i\}_{i=1}^r$ be an $r$-clique-cover of $G$, and
	let $\calB$ be the set of maximal cliques in $G$ whose corresponding nodes in $T_G$ have degree greater than two.
	By Claim~\ref{claim:leaves_are_X},  the leaves of $T_G$ are a subset of $\calC$, 
and thus there are at most $r$ of them. As $T_G$ is a tree, there can be at most $r$ internal nodes of degree greater than 2, so that $|\calB| \leq r$.
Let $\calS = \calB \cup \calC$ 
and $k = |\calS|$ so that $k \leq 2r$.
	
	Let $\calH = \{(C_j, C'_j)\}$ be the set of pairs of 
maximal cliques in $\calS$ that have a non-branching path between them in $T_G$ that
does not contain any maximal clique in $\calC$.
As $T_G$ is a tree, $|\calH| = k-1$.
	We define a series of $k$ graphs, such that $G_0 = G$ and $G'=G_k$, where $G_{j+1}$ is obtained from $G_j$ by preforming a $(C_j, C'_j, \eps')$-simplification (which is possible by Lemma~\ref{lem:two_clique_partition}).
	
	By Claim~\ref{clm:tree_is_the_same_after_fix}, each graph $G_j$ is chordal and $r$-clique-coverable. Furthermore, there exists a clique tree $T_{G_j}$ such that the set of maximal cliques of degree greater than $2$ in $T_{G_j}$ is $\calB$, and $\calC$ is an $r$-cover of $G_j$.
In addition, for every pair $(C_\ell, C'_{\ell})$ where $\ell \leq j$, we have that
the length of a non-branching path between $C_\ell$ and $C'_\ell$
is at most $2/\eps'$.
As $|\calH| \leq 2r$, we get that the number of maximal cliques in  $T_{G'}$ is bounded by $4r/\eps'$, as claimed.

Since each $G_j$ is a $(C_j,C'_j,\eps')$-simplification of $G_{j-1}$, and each simplification is applied to a different pair of maximal cliques, Item~\ref{item:x_part:close_cliques} of the current lemma follows as well by Definition~\ref{def:simplification}.

It remains to 
upper bound the number of edges modified in the transformation of $G$ to $G'$.
We have that
	$$\Delta(G,G') \leq  \sum_j  \Delta(G_j, G_{j+1}) \leq  2r\cdot \eps' \;.$$
	The first inequality is due to the triangle inequality, and the second is due to Item~\ref{item:x_part:close_cliques} in Definition~\ref{lem:two_clique_partition}.
%
%
\end{proof}

\subsection{Proofs of Lemma~\ref{lem:P_accepts_trees} and Lemma~\ref{lem:strict_prop_accepts_only_chordal}}
\label{subsec:P-T-lemmas}

The proof of Lemma~\ref{lem:P_accepts_trees} is relatively simple, and relies on the next claim.

\begin{claim} \label{clm:cliques_with_cut_are_CC}
	Let $G$ be a chordal graph with clique tree $T_G$. Let $v$ be a vertex in $G$, and let $\calC(v)$ be the set of maximal cliques in $G$ that contain $v$. Then the subgraph of $T_G$ induced by $\calC(v)$ is connected.
\end{claim}
\begin{proof}
	Assume by way of contradiction 
that $T_G[\calC(v)]$ is not connected.
Let $C_1, C_2$ be two cliques in different components of $T_G[\calC(v)]$.
Then the shortest path between $C_1$ and $C_2$ in $T_G$ contains some maximal clique $C_3$ that does not belong to $\calC(v)$.
But by the definition of a clique tree, $C_1 \cap C_2 \subseteq C_3$, implying that $v\in C_3$, and we reached a contradiction
(since $C_3 \notin \calC(v)$).
\end{proof}


\begin{proof}[Proof of Lemma~\ref{lem:P_accepts_trees}]
	Let $G$ be a graph with clique tree $T_G$. We assign each vertex $v$ of $G$ to a part  in $\calK(T_G)$,
which we denote by $K(v)$,
as follows.
Let $\calC(v)$ be the set of maximal cliques that $v$ belong to.
	By Claim~\ref{clm:cliques_with_cut_are_CC}, the subgraph induced by $\calC(v)$ in $T_G$ is connected.
We let $K(v)$ be the part corresponding to this connected subgraph.

	We now verify that this assignment does not violate any constraints defined by $\phi_{T_G}$.
Let $u$ and $v$ be two vertices in $G$. If there is an edge between $u$ and $v$, then there is a maximal clique in $G$ to which
both $u$ and $v$ belong.
Therefore, $\calC(u) \cap \calC(v) \neq \emptyset$, so that $K(u)$ and $K(v)$ intersect and $\phi_{T_G}(K(u),K(v)) = 1$.
	If $u$ and $v$ do not have an edge between them, then there is no clique that contains both of them.
Therefore, $\calC(u) \cap \calC(v) = \emptyset$, so that $K(u)$ and $K(v)$ do not intersect and $\phi_{T_G}(K(u),K(v)) = 0$.
\end{proof}

We prove Lemma~\ref{lem:strict_prop_accepts_only_chordal} using the following definition and theorem of 
Gavril~\cite{Gavril}.

\begin{definition}\label{def:subtree-graph}
	A graph $G$ is called a {\sf subtree graph} if there exists a tree $T$ and a  mapping $M$ from the vertices of $G$ to connected subgraphs of $T$ 
(subtrees),
so that every two vertices $u$ and $v$ have an edge in $G$ if and only if their mapped subtrees $M(u)$ and $M(v)$ intersect.
\end{definition}

\begin{theorem}[{\cite[Thm.~3]{Gavril}}]\label{thm:Gavril}
	A graph $G$ is chordal if and only if it is a subtree graph.
\end{theorem}

\begin{proof}[Proof of Lemma~\ref{lem:strict_prop_accepts_only_chordal}]
	Let $T$ be a tree, and let $G=(V,E)$ be a graph in $\calP_{\phi_T}$.
Let $\calV = (V_1,\dots,V_t)$ be  a witness partition of $G$ for  $\phi_T$. For any vertex $v$ let $\ind_\calV(v)$ denote the index of the part that $v$ belongs to in $\calV$.
	Let $M_T: V \rightarrow \calK(T)$ be the mapping that maps each vertex $v\in V$ to the subtree of $T$ that defines part $\ind_\calV(v)$ (recall Definition~\ref{def:chordal-partition}).
	We complete the proof by observing that by their respective definitions, $T$ and $M_T$ are a tree and a mapping as described in Definition~\ref{def:subtree-graph}. Therefore $G$ is a subtree graph, and by  Theorem~\ref{thm:Gavril} is chordal.
\end{proof}

\addcontentsline{toc}{section}{References}

\bibliography{dist-est-ref}

\appendix
\ifnum\fullver=0

\section{Proof of Theorem~\ref{thm:cover}}\label{app:cover}

In this appendix we prove the ``Cover Theorem'' from Section~\ref{sec:cover}.

\begin{proof}[Proof of Theorem~\ref{thm:cover}]
\ProofOfCoverTheorem
\end{proof}

\fi
\section{Semi-homogeneous partition properties}\label{app:shpp-csp}

In this section we prove Lemma~\ref{lem:shpp-are-ee} (stated in Section~\ref{subsubsec:intro-SHPP}).
We first introduce some notations.
For a given integer $k$, a partition function $\phi: [k]\times [k] \to \{0,1,\bot\}$,  a partition $\calV = (V_1,\dots,V_k)$ of $V = V(G)$,
and a pair of (distinct) vertices $u \in V_i$ and $v\in V_j$, we say that the pair $(u,v)$ is a \emph{violating pair} with respect to $\phi$ and $\calV$ if $\phi(i,j) = 0$ and $(u,v) \in E$ or $\phi(i,j) = 1$ and $(u,v) \notin E$.
We use $\mu_\phi(\calV)$ to denote the number of such violating pairs, normalized by $n^2$ (where $n= |V|$),
and let $\mu_\phi(G)$ denote the minimum value of $\mu_\phi(\calV)$ taken over all $k$-partitions $\calV$.
Observe that $\Delta(G,\calP_\phi) = \mu_\phi(G)$.
We denote by $\bar{\mu}_\phi(\calV)$ the number of non-violating pairs, normalized by $n^2$.
Note that as the total number of vertex pairs (violating and non-violating) normalized by $n^2$ is $\binom{n}{2} / n^2 = (1-1/n)/2$, and thus $\bar{\mu}_\phi(\calV) = (1-1/n)/2 - \mu_\phi(\calV)$. Therefore by approximating 
$\max_\calV \left\{\bar{\mu}_\phi(\calV)\right\} $
we can derive an approximation of
$\min_\calV \left\{\mu_\phi(\calV)\right\}
  = \mu_\phi(G) = \Delta(G, \calP_\phi)$ (where in both cases we mean an \emph{additive} approximation).

We 
show how  the problem of maximizing $\bar{\mu}_\phi(G)$ can be casted as a 
maximum constraint satisfaction problem, and approximate the value of an optimal solution 
using the work 
of Andersson and Engebretsen~\cite{AE02}.\footnote{The algorithm given in~\cite{AE02} is actually a variation (and generalization) of the algorithm described in~[Sec.~8]\cite{GGR}, for testing related partition problems.}

First we 
quote some needed definitions and results from~\cite{AE02}, and then show how 
Lemma~\ref{lem:shpp-are-ee} follows.

\begin{definition}
	Let $D$ be some finite domain.
An $r$-ary {\sf constraint function} on domain $D$ is a function from $D^r$ to $\{0,1\}$ nd
%
an $r$-ary {\sf constraint family} on domain $D$ is a collection of $r$-ary constraint functions on domain $D$.
\end{definition}

In the following definitions, $\calF$ is an $r$-ary constraint family on domain $D$.
\begin{definition}
The maximum number of simultaneously satisfiable constraint functions in $\calF$, denoted by $\Sigma(\calF)$, is defined by
	\begin{equation*}
		\Sigma(\calF) = \underset{(a_1, \dots , a_r)\in D^r}{\max}{|\{f\in \calF : f(a_1, \ldots, a_r)=1\}}
	\end{equation*}
\end{definition}

\begin{definition}
A {\sf constraint} on the variables $x_1,\ldots , x_n$ over $\calF$ is an $(r+1)$-tuple $(f, x_{i_1}, \ldots, x_{i_r})$, where $f\in \calF$ and $i_1, \ldots,i_r$ are distinct integers in $[n]$.
%

The  {\sf Constraint Satisfaction Problem} Max-$\calF$ is the following maximization problem: Given a collection $\calC$ of constraints on the variables $x_1, \dots, x_n$ over $\calF$, find an assignment to those variables that satisfies as many constraints as possible in $\calC$.
\end{definition}

\begin{definition}
Let 
$\calC$ be an instance of Max-$\calF$. The {\sf density} of 
$\calC$, denoted $\rho(\calC)$,
is defined as the maximum number of satisfied constraints in 
$\calC$
divided by $n^r$.
\end{definition}

Andersson and Engebretsen~\cite{AE02} consider algorithms that approximate $\rho(\calC)$ for instances $\calC$ of Max-$\calF$ when
given \emph{query access} to $\calC$. A single query to $\calC$ asks whether a constraint $(f, x_{i_1}, \ldots x_{i_r})$ belongs to $\calC$.
The following theorem states that there exists an 
algorithm that approximates Max-$\calF$ using only a small number of queries to $\calC$ (for the sake of succinctness the theorem we present here is a slightly modified version of the one in~\cite{AE02}).
\begin{theorem}[{\cite[Thm.~2]{AE02}}]\label{thm:const_sat}
	Let $\calF$ be an $r$-ary constraint family on domain $D$ and let $\calC$ be an instance of Max-$\calF$. There exists an algorithm $\hat{A}_\calF$ that for any $\eps>0$ and $\delta>0$, outputs a value $\hat{\rho}$ such that $|\hat{\rho}-\rho(\calC)| \leq \eps$ with probability at least $1-\delta$.
The query complexity of $\hat{A}_\calF$ is
	\begin{equation*}
	O\left(\frac{|\calF|\cdot \Sigma^7(\calF) \cdot r^2 \cdot \ln(|D|)}{\eps^7}\cdot \ln^2\left(\frac{\Sigma(\calF)\cdot |D|}{\eps \cdot \delta}\right)\right)\;.
	\end{equation*}
Furthermore, for $r=2$ the algorithm takes a sample of
	$O\left(\frac{\Sigma^5(\calF)\cdot \ln(|D|)}{\eps^5}\cdot \ln\left(\frac{\Sigma(\calF)\cdot |D|}{\eps \cdot \delta}\right)\right)$
uniformly selected variables and performs queries on a subset of the constraints involving pairs of sampled variables.
\end{theorem}


\begin{proof}[Proof of Lemma~\ref{lem:shpp-are-ee}]
As discussed at the start of this section, in order to approximate $\Delta(G,\calP_\phi) = \mu_\phi(G)$
to within an additive error of $\eps$, it suffices to approximate $\bar{\mu}_\phi(G)$ to within such an error.
We show how, for any $k$-part
partition function $\phi$, we can define a $2$-ary constraint family $\calF_\phi$, and for any graph $G$ we can
define an instance $\calC_{\phi, G}$ of Max-$\calF_\phi$, such that $\bar{\mu}_\phi(G) = \rho(\calC_{\phi,G})$.
Furthermore, any query to $\calC_{\phi,G}$ can be answered by performing a single query to (the adjacency matrix of) $G$.
We can therefore run the algorithm referred to in Theorem~\ref{thm:const_sat} and obtain a distance-approximation algorithm for
$\calP_\phi$.
Details follow.

\smallskip
Let $D_{\phi} =[k]$ 
and let $f^1_{\phi}$ and $f^0_{\phi}$ be two $2$-ary functions on domain $D_\phi$, defined as follows: $f^1_{\phi}(y_1, y_2) = 1$ if $\phi(y_1, y_2) \in \{1, \bot\}$ and $f^1_{\phi}(y_1, y_2) =0$ otherwise; $f^0_{\phi}(y_1, y_2) = 1$ if $\phi(y_1, y_2) \in \{0, \bot\}$ and $f^0_{\phi}(y_1, y_2) =0$ otherwise. We set the constraint family to be $\calF_\phi = \{f^0_\phi, f^1_\phi\}$.
For a graph $G= (V,E)$, let $X_{G} = \{x_v | v\in V\}$ be a set of variables, and
let
$$\calC_{\phi, G} = \left\{(f^1, x_v, x_u) | (v,u) \in E\right\} \cup \left\{(f^0, x_v, x_u) | (v,u) \not\in E\right\}$$
 be a collection of constraints on the variables $X_{G}$ over $\calF_\phi$.
Note that by the above definition of $\calC_{\phi, G}$, any query to $\calC_{\phi, G}$
(i.e., does $(f^b,x_u,x_v)$ belong to $\calC_{\phi,G}$ for $b \in \{0,1\}$ and $u,v \in V$), can be 
answered by performing a single query to (the adjacency matrix of) $G$ (on the pair of vertices $u,v$).
%
Observe that $r=2$, $|D_\phi| = k$, $|\calF_\phi| =2$, and $\Sigma(\calF_\phi)=1$ (as every pair of variables appears in exactly one constraint).
Also recall that each variable in $X_G$ corresponds to a vertex in $G$ (and each query to $\calC_{\phi, G}$ can be answered by
a query to $G$). Hence, Theorem~\ref{thm:const_sat} implies that by selecting
a sample of $\poly(1/\eps,\log k,\log(1/\delta))$
 vertices in $G$ (uniformly, independently at random) and querying the subgraph induced by the sample,
 we can get an estimate $\hat{\rho}$ such that $|\hat{\rho} - \rho(\calC_{\phi,G})| \leq \eps$
 with probability at least $1-\delta$.

\smallskip
It remains to show that $\rho(\calC_{\phi,G}) = \bar{\mu}_\phi(G)$.
For any vertex $v\in V$ and partition $\calV = (V_1,\dots,V_k)$ of $V$, let $\ind_\calV(v)$ denote the index of the part that $v$ belongs to in $\calV$. 
An assignment to the variable set $X_G$ 
is a function $g: X_G\rightarrow D_\phi$.

To show the 
$\rho(\calC_{\phi,G}) = \bar{\mu}_\phi(G)$,
we define a bijective mapping $M$ between partitions of $V$ and assignments to the variables in $X_{G}$, and show that $\bar{\mu}_\phi(\calV)$ is equal to the number of constraints in $\calC_{\phi, G}$ satisfied by the assignment $M(\calV)$,
divided by $n^2$.
The mapping is as follows. For any partition $\calV$ of $G$, $M(\calV)$ is the  assignment to $X_{G}$ satisfying $M(\calV)(x_v) = \ind_\calV(v)$ (similarly $M^{-1}(g)$ is the partition that puts vertex $v$ in part $g(x_v)$). 

Observe that any pair of variables $x_u,x_v$ appear in exactly one constraint. If $(u,v) \in E$, then this constraint is $(f^1, x_u, x_v)$, and thus it is satisfied if and only if $\phi(\ind_\calV(v), \ind_\calV(u)) \in \{1, \bot\}$. Similarly,  if $(u,v) \notin E$, then the constraint on the variable pair $x_u, x_v$ is satisfied if and only if $\phi(\ind_\calV(v), \ind_\calV(u)) \in \{0, \bot\}$. In both cases $u,v$ is a non-violating pair (with respect to $\phi$ and $\calV$) if and only if the constraint involving $x_u$ and $x_v$ is satisfied. Therefore the number of satisfied constraints in the assignment 
defined by the partition $\calV$ is equal to $\bar{\mu}_\phi(\calV)\cdot n^2$, as claimed.
Since $\rho(\calC_{\phi,G})$ is the maximum number of constraints that can be satisfied by any assignment to the variables in $X_G$, divided by $n^2$, the lemma follows.
\end{proof}
\section{General partition properties}
\label{app:gpp}

In this section we show how to obtain a distance-approximation algorithms for general partition properties (defined next and introduced in~\cite{GGR}).
We note that it is possible to
obtain essentially the same result by applying Theorem~3.2 in~\cite{HKLLS17a}. This theorem gives a general upper bound on the query complexity of approximating a certain class of graph parameters.
This class of graph parameters includes the distances to properties in a subfamily of general partition properties. In turn, this subfamily
can be shown to cover all general partition properties. We give an alternative proof, which we believe is simpler.

\begin{definition}\label{def:gpp}
	For an integer $k \geq 1$, a {\sf $k$-part partition description} $\Phi$ is  defined by two $[k]\rightarrow [0,1]$ functions, denoted $\SLB{\Phi}$ and $\SUB{\Phi}$, and two $[k]\times[k]\rightarrow [0,1]$ symmetric functions, denoted $\DLB{\Phi}$ and $\DUB{\Phi}$ such that
	$\SLB{\Phi}(i)\leq \SUB{\Phi}(j)$ for every $i \in [k]$ and
	$\DLB{\Phi}(i,j)\leq \DUB{\Phi}(i,j)$ for every $i,j \in [k]$.
	For a graph $G = (V,E)$ we say that a $k$-way partition $\calV = (V_1,\dots,V_k)$ {\sf satisfies} $\Phi$
	(with respect to $G$) if
	\begin{equation} \label{eq:gpp-prop1}
	\forall i\in[k]\;\;\;\;\;\;
	\SLB{\Phi}(i)\cdot n \;\leq\;|V_i|\;\leq \;\SUB{\Phi}(i)\cdot n\;
	\end{equation}
	and
	\begin{equation} \label{eq:gpp-prop2}
	\forall i,j\in[k]\;\;\;\;\;\;
	\DLB{\Phi}(i,j)\cdot n^2\;\leq\;|E_G(V_i,V_{j})|\;\leq\;\DUB{\Phi}(i,j)\cdot n^2\;,
	\end{equation}
	where $E_G(V_i,V_{j})$  is the set of edges in $G$ between vertices in
	$V_i$ and vertices in $V_{j}$, and $n$ is the number of vertices in $G$.
	We denote by $\calP_\Phi$  the set of graphs $G$ for which there is a partition $\calV$ that satisfies $\Phi$,
	and we use 
	$\GPP$ to denote the family of all such partition properties.
\end{definition}

Following is the central theorem of this section.
\begin{theorem} \label{thm:general_ggr_EE}
	For any $k$-part partition description $\Phi$, there exists a distance-approximation algorithm $\calA$  for the property $\calP_\Phi$ with query complexity $\poly(k,1/\eps)$. 
\end{theorem}

In what follows we refer to $\SLB{\Phi}$ and $\SUB{\Phi}$ as the \emph{size} functions 
and to $\DLB{\Phi}$ and $\DUB{\Phi}$ as the (absolute) \emph{density} functions.
We note that in~\cite{GGR}, in the constraints imposed by the density functions, each edge was counted twice (corresponding to the two entries in the adjacency matrix). For the sake of simplicity, we have chosen to count each edge once
This does not really have any effect on the analysis.

We rely on Theorem~2.7 from~\cite{FMS}, which generalizes (to hypergraphs) and improves (in terms of the query complexity) the result of~\cite{GGR} on testing general partition properties. Below we state the theorem for the case of graphs, which suffices for our purposes.

%
\begin{theorem}[Special case of~{\cite[Thm. 2.7]{FMS}}] \label{gen-par-test.thm}
	There exists an algorithm $\calA$ such that for any given $k$-part partition description $\Phi$,
	algorithm $\calA$ is a property testing algorithm for the
	property $\calP_\Phi$
	with query complexity
$\poly(k,1/\eps,\log(1/\delta))$,
	that errs with probability at most $\delta$.

\end{theorem}

We also use the fact that algorithm $\calA$ in the above theorem is non-adaptive and its queries depend only on $\eps, \delta, k$ and not on $\Phi$.

\begin{definition}
For $\gamma \in [0,1]$, we say that a $k$-part partition description $\Phi$ is
	{\sf $\gamma$-tight} if
$\SUB{\Phi}(i)-\SLB{\Phi}(i) \leq \gamma$ for every $i\in [k]$
	and $\DUB{\Phi}(i,j)-\DLB{\Phi}(i,j)  \leq \gamma$ for every $i,j \in [k]$.
%
\end{definition}

We shall 
make use of the following notations.
For a graph $G$ and a partition $\calV = (V_1,\dots, V_k)$ of $V$, let $\SP{\calV}$ be a $[k]\rightarrow [0,1]$ function such that $\SP{\calV}(i) = \frac{|V_i|}{n}$, and let $\DP{\calV}{G}$ be a $[k]\times[k] \rightarrow [0,1]$ symmetric function such that $\DP{\calV}{G}(i,j)=\frac{|E_G(V_i, V_j)|}{n^2}$. We refer to $\SP{\calV}$ as the \emph{size function} of $\calV$, and to $\DP{\calV}{G}$ as the \emph{density function} of $\calV$ on graph $G$ (when clear from context, the subscript $G$ will be omitted).
For two functions $f_1$ and $f_2$ over the same domain $X$, we let $\Lone{f_1}{f_2} = \sum_{x\in X}|f_1(x) - f_2(x)|$ denote the $L_1$ distance between them. 


For any $\gamma\in [0,1]$ and $k$-part partition description $\Phi$, let $\TPD{\Phi}{\gamma}$ be the set of $\gamma$-tight $k$-part partition descriptions $\Phi'$ such that $\SLB{\Phi'}(i) = \SLB{\Phi}(i)$ and $\SUB{\Phi'}(i) = \SUB{\Phi}(i)$ for every $i\in [k]$, and such that
the range of the density functions $\DLB{\Phi'}$ and $\DUB{\Phi}$  consist of
%
integer multiples of $\gamma$. 

In all that follows, $k$ always denotes an integer (the number of parts in a partition description or partition), and $\gamma$ is always a fraction in $[0,1]$.

We start by making two simple observations that will be used in the proofs of the main lemmas.
\begin{observation}\label{obs:dist_between_densities_same_partition}
	Let $G_1=(V,E_1)$ and $G_2=(V, E_2)$ be two graphs over the same set $V$ of $n$ vertices and let
$\calV^1$ and $\calV^2 $ be two partitions of $V$
such that $\SP{\calV^1} = \SP{\calV^2}$.
	Let $G'_1 = (V, E'_1)$ be a graph that minimizes $|E_1 \SymDif E'_1|$ subject to
	$\DP{\calV^1}{G'_1} = \DP{\calV^2}{G_2}$.
	Then 
$\Delta(G_1,G'_1)= \RPDist{\calV^1}{\calV^2}{G_1}{G_2}$.
\end{observation}
\begin{proof}
Let $\calV^1 = (V^1_1,\dots,V^1_k)$ and $\calV^2 = (V^2_1,\dots,V^2_k)$.
	For any $i,j\in [k]$, the minimal number of edge modifications that must be performed on $G_1$ so as to obtain
a graph $G_1'$ in which  $|E_{G'_1}(V^1_i, V^1_j)| = |E_{G_2}(V^2_i,V^2_j)|$
is simply $\left||E_{G_1}(V^1_i, V^1_j)| - |E_{G_2}(V^2_i, V^2_j)|\right|$.
	As each vertex pair can only cross one cut, summing over all pairs $i,j \in [k]$ gives the required result.
\end{proof}

\begin{observation} \label{obs:l1_dist_in_tight_intervals}
Let $\Phi_1$ and $\Phi_2$ be two $k$-part $\gamma$-tight partition descriptions, and let $f_1,f_2: [k]\times[k] \to [0,1]$ be two density functions 
such that for $b \in \{1,2\}$ and every $i,j \in [k]$,~
$\DLB{\Phi_b}(i,j) \leq f_b(i,j) \leq \DUB{\Phi_b}(i,j)$.
Then $|\Lone{f_1}{f_2} -\PDDist{\Phi_1}{\Phi_2}| \leq   2\gamma\cdot k^2$.
\end{observation}
\begin{proof}
Since $\DLB{\Phi_b}(i,j) \leq f_b(i,j) \leq \DUB{\Phi_b}(i,j)$ for each $b \in \{1,2\}$ and $i,j \in [k]$,
and both $\Phi_1$ and $\Phi_2$  are $\gamma$-tight, we get that
$\Lone{f_b}{\DLB{\Phi_b}} \leq  \gamma \cdot k^2$ for each $b \in \{1,2\}$.
	The observation now follows from the triangle inequality.
\end{proof}


\begin{lemma}\label{lem:dist_for_phi_works}
	Let $\Phi_1$ and $\Phi_2$ be two $k$-part $\gamma$-tight partition descriptions 
whose size functions are equal and the property $\calP_{\Phi_2}$ is not empty.
	Then for any graph $G_1 \in \calP_{\Phi_1}$, $\Delta(G_1, \calP_{\Phi_2}) \leq \PDDist{\Phi_1}{\Phi_2}+4\gamma \cdot k^2$.
\end{lemma}
\begin{proof}
	Let $n$ be the number of vertices in $G_1$, and
	let $G_2$ be some graph in $\calP_{\Phi_2}$ with $n$ vertices.
	Let $\calV^1 = (V^1_1, \dots, V^1_k)$ and $\calV^2 =(V^2_1, \dots, V^2_k)$ be partitions of $V$ 
such that $\calV^1$ satisfies $\Phi_1$ with respect to $G_1$, and $\calV^2$ satisfies $\Phi_2$ with respect to $G_2$.

	We next show  that there  exists a graph $\tilde{G}_1$, such that
	 $\Delta(G_1, \tilde{G}_1) \leq \PDDist{\Phi_1}{\Phi_2}+4\gamma \cdot k^2$,
and there exists a partition $\tilde{\calV}^1$ of $V$
such that $\SP{\tilde{\calV}^1} = \SP{\calV^2}$ and $\DP{\tilde{\calV}^1}{\tilde{G}_1} = \DP{\calV^2}{G_2}$. This completes the proof as $\tilde{G}_1 \in \calP_{\Phi_2}$.
	
	
	Let $\tilde{\calV}^1= \{\tilde{V}^1_1,\dots, \tilde{V}^1_k\}$ be a partition of 
$V$ such that $\sum_i |\tilde{V}^1_i\SymDif V^1_i|$ is minimized under the constraint that $|\tilde{V}^1_i| = |V^2_i|$ for every $i\in [k]$.
	As $\max_{i \in [k]} \left||V^1_i| - |V^2_i|\right| \leq \gamma \cdot n$,
the partition $\tilde{\calV}^1$ can be constructed from $\calV^1$ by changing the assignment to parts of less than $\gamma \cdot k \cdot n$ vertices. Since each vertex participates in at most $n$ vertex pairs, 
	\begin{equation} \label{eq:dimn_dist}
		\RPDist{\tilde{\calV}^1}{\calV^1}{G_1}{G_1} \leq \frac{2 \gamma \cdot k \cdot n^2}{n^2} = 2 \gamma \cdot k\;.
	\end{equation}
	As $\calV^1$ satisfies $\Phi_1$ with respect to $G_1$ and $\calV^2$ satisfies $\Phi_2$ with respect to $G_2$, by applying  Observation~\ref{obs:l1_dist_in_tight_intervals} we have that
	\begin{equation}\label{eq:v1_v2}
	\RPDist{\calV^1}{\calV^2}{G_1}{G_2} \leq \PDDist{\Phi_1}{\Phi_2} + 2\gamma \cdot k^2\;.
	\end{equation}
By applying Observation~\ref{obs:dist_between_densities_same_partition}, we have that there exists a graph $\tilde{G}_1$
such that $\tilde{\calV}^1$ satisfies $\Phi_2$ with respect to $\tilde{G}_1$, and
	\begin{equation}\label{eq:G1-tG1}
	\Delta(G_1,\tilde{G}_1) = \RPDist{\tilde{\calV}^1}{\calV^2}{G_1}{G_2}\;.
	\end{equation}
	Combining Equations~\eqref{eq:dimn_dist}--\eqref{eq:G1-tG1} with the triangle inequality, gives us that
$$\Delta(G_1,\tilde{G}_1) \leq \RPDist{\tilde{\calV}^1}{\calV^1}{G_1}{G_1} + \RPDist{\calV^1}{\calV^2}{G_1}{G_2}
  \leq \PDDist{\Phi_1}{\Phi_2} + 4 \gamma \cdot k^2 \;,
$$
as claimed.
\end{proof}


\begin{lemma} \label{lem:exists_partition_description_dist_lower_bound}
For any graph $G$ and $k$-part  $\gamma$-tight  partition description $\Phi$ such that $\calP_\Phi$ is not empty,
	there exists a partition description $\Phi'\in \TPD{\Phi}{\gamma}$  such that $G\in \calP_{\Phi'}$ and $\PDDist{\Phi'}{\Phi} \leq \Delta(G, \calP_\Phi) + 2\gamma \cdot k^2$.
\end{lemma}
\begin{proof}
	Let $G'= (V,E')$ be a graph in $\calP_\Phi$ that is closest to $G= (V,E)$ and let $\calV = (V_1, \dots, V_k)$ be
a partition of $V$ that satisfies $\Phi$ with respect to $G'$.
Let $\Phi'\in \TPD{\Phi}{\gamma}$ be such that for all $i,j\in [k]$ the density of the cut $(V_i, V_j)$ in $G$ is in the range specified by $\Phi'$ (recall that the size functions of all descriptions in $\TPD{\Phi}{\gamma}$ are equal to those of $\Phi$).
	As $G'$ is a graph closest  to $G$ such that the partition $\calV$ satisfies $\Phi$ with respect to $G'$, by Observation~\ref{obs:dist_between_densities_same_partition} $\Delta(G, G') = \RPDist{\calV}{\calV}{G}{G'}$.
	By Observation~\ref{obs:l1_dist_in_tight_intervals} $\PDDist{\Phi}{\Phi'} \leq \RPDist{\calV}{\calV}{G}{G'} + 2\gamma k^2$.
	The lemma follows by combining the above two equations.
\end{proof}

Before proving Theorem~\ref{thm:general_ggr_EE}, we prove the simpler case of estimating the distance to partition properties defined by tight partition descriptions.
\begin{lemma}\label{lem:tight-gpp-are-ee}
	There exists an algorithm $\calA_{\ref{lem:tight-gpp-are-ee}}$ such that for any $\eps>0$, $\delta>0$,
	$k$-part $(\eps/(5k^2))$-tight  partition description $\Phi$, and graph $G$,
	algorithm $\calA_{\ref{lem:tight-gpp-are-ee}}$ returns  an estimate $\hat{\Delta}$ such that
$\left|\hat{\Delta} - \Delta(G, \calP_\Phi)\right| \leq \eps$
with probability at least $1-\delta$.
The query complexity of $\calA_{\ref{lem:tight-gpp-are-ee}}$ is
	$$\poly(\log(1/\delta), k, 1/\eps)$$
and its queries only depend on $k$, $\eps$ and $\delta$.
\end{lemma}

\begin{proof}
	Denote $\eps'=\eps/5k^2$.
	Given input graph $G$, our algorithm tests whether $G \in \calP_{\Phi'}$ with distance parameter $\eps'$ and confidence parameter $\delta \cdot (\eps')^{k^2}$, for each $\Phi' \in \TPD{\Phi}{\eps'}$, using the testing algorithm of Theorem~\ref{gen-par-test.thm}.
	Let $\Sd$ be the set of partition descriptions in $\TPD{\Phi}{\eps'}$ for which the algorithm return ``Accept'' (so that we are ensured with high probability that $G$ is close to or in the corresponding properties).
	Let 
$\hat{\Phi} = \argmin_{\Phi' \in \Sd} \left\{\PDDist{\Phi'}{\Phi}\right\}$
be the property closest to $\Phi$ among the properties $\Phi' \in \Sd$. The algorithm returns $\hat{\Delta} = \PDDist{\hat{\Phi}}{\Phi}$.
	
	Since the testing algorithm is non-adaptive and its queries do not depend on $\Phi'$, we can run 
the algorithm for all $\Phi' \in \TPD{\Phi}{\eps'}$ using the same queries. Thus, by Theorem~\ref{gen-par-test.thm}, we have the desired query complexity.

By its definition $|\TPD{\Phi}{\eps'}| \leq \frac{1}{(\eps')^{k^2}}$  , and thus we run the testing algorithm for at most $\frac{1}{(\eps')^{k^2}}$ partition descriptions.
Using a union bound on the probability of failure in any if these tests, we get that with probability at least $1-\delta$ the testing algorithm succeeds in all of them. Namely, with probability at least $1-\delta$, for every $\Phi' \in \TPD{\Phi}{\eps'}$, if $G\in \Phi'$, then $\Phi' \in \Sd$, and if $\Delta(G,\calP_{\Phi'}) > \eps'$, then $\Phi' \notin \Sd$.
We henceforth condition on this event, and show that this implies that 
$\left|\hat{\Delta} - \Delta(G, \calP_\Phi)\right| \leq \eps$

Since $\Delta(G, \calP_{\hat{\Phi}})\leq \eps'$,
by Lemma~\ref{lem:dist_for_phi_works} and the triangle inequality we have that 
\begin{equation} \label{eq:lower}
	\Delta(G, \calP_\Phi) \leq \hat{\Delta} + 4\eps' \cdot k^2 + \eps' \leq \hat{\Delta} + \eps\;.
	\end{equation}
	Let $\Phi^*$ be the partition description implied by  Lemma~\ref{lem:exists_partition_description_dist_lower_bound} when applied with 
$\Phi$, $G$, and $\gamma=\eps'$.
	Then $\Phi^* \in \TPD{\Phi}{\eps'}$ and $G \in \calP_{\Phi^*}$ so that $\Phi^* \in \Sd$.
	By the definition of $\hat{\Delta}$, 
$\hat{\Delta} \leq \PDDist{\Phi^*}{\Phi}$.
By Lemma~\ref{lem:exists_partition_description_dist_lower_bound},
$\PDDist{\Phi^*}{\Phi} \leq \Delta(G, \calP_\Phi) + 2\eps' \cdot k^2 \leq \Delta(G, \calP_\Phi)+\eps$.
Hence,
	\begin{equation} \label{eq:upper}
	\hat{\Delta} \leq \Delta(G, \calP_\Phi) + \eps\;.
	\end{equation}
	By Equations~\eqref{eq:lower} and~\eqref{eq:upper} the algorithm outputs an estimate $\hat{\Delta}$ as required. 
%
%
\end{proof}

Next we observe any $\calGPP$ is covered by $\gamma$-tight $\calGPP$, for any $\gamma>0$, which allows us to use the cover theorem (Theorem~\ref{thm:cover}).
\begin{observation}
	Let $P_\Phi$ be $k$-part $\calGPP$. Then there exists a $0$-cover of $P_\Phi$ of $\gamma$-tight $k$-part partition properties, of size at most $\frac{1}{\gamma}^{k^2+k}$.
\end{observation}

The proof of Theorem~\ref{thm:general_ggr_EE} follows from the above observation and Lemma~\ref{lem:tight-gpp-are-ee}.

\subsection{Relative density constraints}
Recall that general partition properties in $\GPP$ 
impose density constraints on the fraction of edges in a cut (between two parts in a partition), relative to $n^2$.
Nakar and Ron~\cite{NR18} study a generalization of $\GPP$, denoted $\GPPWR$, which allows in addition to  impose constraints on the fraction of edges in a cut \emph{relative to the number of vertex pairs in the cut}.
Note that semi-homogeneous partition properties (the family $\tcalHPP$) are examples of such properties (observe that $\tcalHPP \not \subset \GPP$). 

More formally, recall that properties in $\GPP$ are defined by four functions ($\SLB{\Phi}, \SUB{\Phi}, \DLB{\Phi}, \DUB{\Phi}$ in Definition~\ref{def:gpp}) and impose two constraint sets that valid graph partitions must comply with: Equation~\eqref{eq:gpp-prop1} and~\eqref{eq:gpp-prop2} in Definition~\ref{def:gpp}).
Properties in $\GPPWR$ are defined by two additional $[k]\times [k] \rightarrow [0,1]$ functions, denoted $RD^{lb}_\Phi$ and $RD^{up}_\Phi$, such that $\forall i,j \in [k]$, $RD^{lb}_\Phi(i,j)\leq RD^{ub}_\Phi(i,j)$, and they impose an additional set of constraints:
\begin{equation*}
	\forall i,j\in [k] \;\;\;\; RD^{lb}_\Phi(i,j)\cdot |V_i| \cdot |V_j| \leq  |E_G(V_i, V_j)| \leq RD^{ub}_\Phi(i,j)\cdot |V_i| \cdot |V_j|\;.
\end{equation*}

\noindent In the proof of Theorem~1 in~\cite{NR18} 
it is shown that each $\calP \in \GPPWR$ has an $\eps$-cover of size $\frac{1}{\eps^{O(k^2\cdot \log(k))}}$ by a family of functions that belong to $\calGPP$.
Applying this observation with Theorem~\ref{thm:cover} and Theorem~\ref{thm:general_ggr_EE}, we deduce that properties in $\GPPWR$ also have distance-approximation algorithms with query complexity of $\poly(k, 1/\eps)$.




\ifnum\fullver=0
\section{Missing proofs for Section~\ref{sec:p4}}\label{app:p4}

\subsection{Proof of Lemma~\ref{lem:p4-to-tree}}
\else
\section{Proof of Lemma~\ref{lem:p4-to-tree}}\label{app:p4}
\fi

The proof of Lemma~\ref{lem:p4-to-tree} is based on the next theorem from~\cite{Sei74}, where
for a graph $G = (V,E)$, we let $\overline{G} = (V,\overline{E})$ denote the graph complementary to $G$ (that is, for every pair of vertices $u,v$, we have that $(u,v) \in E$  if and only if $(u,v) \notin \overline{E}$).

	\begin{theorem}[\cite{Sei74}] \label{thm:p4_connectivity}
		Let $G$ be a graph with vertex set $V$. The following statements are equivalent:
		\begin{enumerate}
			\item $G$ has no induced subgraph isomorphic to $P_4$.\label{it:free}
			\item \label{it:sub-con}
For every subset $U \subseteq V$, either $G[U]$ is connected or $\overline{G}[U]$ is connected.
		\end{enumerate}
	\end{theorem}

\begin{proof}[Proof of Lemma~\ref{lem:p4-to-tree}]
Assume first that Item~\ref{it:sub-con} of Theorem~\ref{thm:p4_connectivity} holds.
We show that Item~\ref{it:T-G} of Lemma~\ref{lem:p4-to-tree} holds as well.
Consider the following recursive procedure that defines a tree $T(G)$ as stated in Item~\ref{it:T-G} of Lemma~\ref{lem:p4-to-tree}.
Starting with $U=V$, at each recursive call we have a subset of vertices $U\subseteq V$. If $|U|=1$, then we have reached a leaf of $T(G)$, and the single vertex in $U$ corresponds to this leaf.
Otherwise ($|U|>1$),
we partition $U$ into two disjoint (and non-empty) subsets $U_1$ and $U_2$, such that $(U_1,U_2)$ is either an empty or complete cut in $G$, thus defining an internal node in $T(G)$.
This is always possible, since by  Item~\ref{it:sub-con} of Lemma~\ref{lem:p4-to-tree}, one of the following holds: (1) $G[U]$ is not connected (in which case we can partition the connected components of $G$ into two subsets and obtain an empty cut); (2)  $\overline{G}[U]$ is not connected (in which case we can partition the connected components of $\overline{G}$ into two subsets and obtain a complete cut in $G$).
We then continue recursively with $U_1$ and $U_2$.

Assume next that Item~\ref{it:T-G} of Theorem~\ref{thm:p4_connectivity} holds. We show that
Item~\ref{it:sub-con} of Theorem~\ref{thm:p4_connectivity} holds as well. Let $U$ be any subset of $V$.
Consider the leaves of $T(G)$ that correspond to the vertices in $U$ and let $y$ be the lowest common ancestor of all these leaves.
If the cut $(W_1,W_2)$ corresponding to $y$ is empty, then $G[U]$ is not connected (as both $W_1\cap U \neq \emptyset$ and $W_2\cap U\neq \emptyset$). Similarly, if $(W_1,W_2)$ is complete, then $\overline{G}[U]$ is not connected.
\end{proof}

\ifnum\fullver=0
\subsection{Proof of Lemma~\ref{lem:close-tree}}

\ProofOfLemCloseTree
\fi

\section{Proof of Lemma~\ref{lem:everything_for_assaf_and_lior}}\label{app:c4}

In order to prove Lemma~\ref{lem:everything_for_assaf_and_lior}, we make use of the following
lemma~\cite[Lem. 3.7]{GS19},
which is slightly rephrased for consistency with our notations). For two sets $S,T$, we use $S\Delta T$ for the symmetric difference between the sets.

\begin{lemma}[\cite{GS19}]\label{lem:cond_regularity}
	There is an absolute constant $c > 0$
	such that for every
	$\alpha,\gamma \in (0,1)$, every $n$-vertex graph $G=(V,E)$ either contains $\Omega(\alpha^c\gamma^c n^4)$ induced copies of $C_4$ or there is a graph $G^{(1)}= (V,E')$, a partition
	$(X_1,\dots,X_k, Y)$ of $V$, where $k \leq 10\alpha^{-3}$, a subset
	$Z \subseteq X := X_1 \cup \dots \cup X_k$,
	a partition $(Q_1,\dots,Q_t)$ of $Q=X \setminus Z$ that refines the partition
        $(X_1\setminus Z,\dots,X_k \setminus Z)$
and subsets $W_i \subseteq Q_i$ for $i\in [t]$, for which the following holds.
	\begin{enumerate}
		\item\label{it:Xi-Z-Y} $G^{(1)}[X_i \setminus Z]$ is a clique for every $i\in [k]$, and $G^{(1)}[Y]$ is an
empty graph.
		\item\label{it:Z-small} $|Z| < \alpha n$ and every $z \in Z$ is an isolated vertex in $G^{(1)}$.
		\item\label{it:sum-Qi-Qj} In $G^{(1)}$, the sum of $|Q_i||Q_j|$, taken over all non-homogeneous pairs $(Q_i,Q_j)$,
		$1 \leq i < j \leq t$, is at most $\alpha n^2$.
		\item\label{it:Wi-Wj} $(W_i,W_j)$ is homogeneous in $G^{(1)}$ for every $1 \leq i < j \leq t$, and
		$|W_i| \geq (\alpha/20)^{4000\alpha^{-6}} |X|$ for every $i\in [t]$.
		\item\label{it:G-G'-close} 
		$\left| E(G^{(1)}) \SymDif E(G) \right|
            < (2\alpha + \gamma)n^2$
            and
		$\left| E(G^{(1)}[X \setminus Z]) \SymDif E(G[X \setminus Z]) \right| < \gamma n^2$.
		\item $\forall i$ $X_i$ is a clique in $G^{(1)}$.
	\end{enumerate}
\end{lemma}

The next lemma is proved very similarly to
Lemma~4.1 in~\cite{GS19}.
It is stated a bit more generally than necessary for our use in the context of $C_4$-freeness, since it also serves us in the proof for chordality. For a family of graphs $\calJ$, we say that a graph $G$ is induced $\calJ$-free, if it is induced $J$-free for every $J \in \calJ$.
\begin{lemma} \label{lem:cut_is_close_to_IS_vertices_cliques}
	Let $\calJ$ be a (finite or infinite) family of graphs such that
	\begin{enumerate}
		\item $C_4 \in \calJ$.
		\item For every $J\in \calJ$ and $v\in V(J)$, the neighborhood of $v$ in $J$ is not a clique.	
	\end{enumerate}
For $\beta >0$
	let $H=(V,E)$ be a graph that is $\beta$-close to being induced $\calJ$-free and such that $V$ can be partitioned into two sets $Q$ and $I$ where $H[Q]$ is induced $\calJ$-free
	and $H[I]$ is empty.
	Then, there exists a graph $H' = (V,E')$ for which the following holds.
	\begin{enumerate}
		\item\label{it:H'-H} $H'$ is $2\beta^{1/4}$-close to $H$.
		\item\label{it:H'-C4-free} $H'$ is induced $\calJ$-free.
		\item\label{it:y-N-CLique} 
		$H'$ differs from $H$ only in the cut $(Q,I)$
		and for every $y\in I$ , $H'[N_{H'}(y)] $ is a clique.
	\end{enumerate}
\end{lemma}
\begin{proof}
Let $\beta' = 2\beta^{1/4}$.
Consider selecting, for each $y \in I$, a maximal {\em anti-matching\/} $M(y)$ in $H[N_Q(y)]$.
That is, a maximal subset of
pairwise-disjoint non-edges contained in $H[N_Q(y)]$
(a {\em non-edge\/} $(u,v)$ is simply a pair of vertices $u \neq v$ such that $(u,v) \notin E$).
For every pair of non-edges $(u,v),(u',v') \in  M(y)$, there must be at least one additional non-edge
with one endpoint in  $\{u,v\}$ and one endpoint in $\{u',v'\}$. This is true since otherwise,
$\{u,v,u',v'\}$ would span an induced $C_4$ in $H[Q]$, in contradiction to the premise of the lemma $H[Q]$ is induced $C_4$-free.
Therefore, for every $y\in I$, there are at least $\binom{|M(y)|}{2} + | M(y)| \geq | M(y)|^2/2$ non-edges in
$G[N_Q(y)]$.
For every $y \in I$ let $d_2(y)$ denote the number of pairs of distinct vertices in $V(N_Q(y))$ that are non-adjacent. Then the above discussion implies that
every $y \in I$ satisfies
\begin{equation}\label{eq:good_triples}
d_2(y) \geq \frac{| M(y)|^2}{2}\;.
\end{equation}

Let $H'$ be the graph obtained from $H$ by deleting, for every $y \in I$, all edges
$(y,v)$ for $v \in V(M(y))$. Observe that by construction, $H'$ satisfies Item~\ref{it:y-N-CLique}
Also observe that since $|V(M(y)| = 2|M(y)|$,
\begin{equation}\label{eq:diff}
|E(H') \SymDif E(H)| = 2\sum_{y \in I}|M(y)|\;.
\end{equation}

We next show that $H'$ is induced $\calJ$-free. Assume, contrary to the claim, that there exists a subset $U\subset V$ such that $H'[U]$ is an induced instance of $J\in \calJ$.
Since by the premise of the lemma, $H[Q]$ is induced $J$-free and since $H'[Q] = H[Q]$, there must be some
$y \in U \cap I$.
Since the neighborhood of $y$ in $J$ is not a clique and since $H'[I] = H[I]$ is an empty graph, there must be $u,v \in U \cap Q$ for which
$u,v \in N_Q(y)$ and $(u,v) \notin E(H')$. Now, the fact that $u,v$ are connected to $y$ in $H'$ means that neither of them participated in
one of the non-edges of $M(y)$. But then the fact that $(u,v) \notin E(H')$ implies that also $(u,v) \notin E(H)$
(because $H'[Q] = H[Q]$),
which in turn implies that $(u,v)$ could have been added to $M(y)$, contradicting its maximality.
Hence, Item~\ref{it:H'-C4-free} is satisfied as well.


It remains to show that Item~\ref{it:H'-H} holds, i.e., $H'$ is $\beta'$-close to $H$.
Assume 
by the way of contradiction,
that $|E(H') \SymDif E(H)| > \beta' n^2$.
Combining this with Equation~\eqref{eq:good_triples}, Equation~\eqref{eq:diff} and Jensen's inequality gives,
\begin{eqnarray}
\sum_{y \in I}d_2(y) &\geq& \frac{1}{2}\sum_{y \in I}{|M(y)|^2}\\
& \geq&
\frac{1}{2}|Y| \cdot \left( \frac{\sum_{y \in I}{|M(y)|}}{|I|} \right)^2 \\
&=&
\frac{1}{2}|I| \cdot \left( \frac{|E(H') \SymDif E(H)|}{2|I|} \right)^2 \\
&\geq&
\frac{1}{8} \frac{(\beta')^2 n^4}{|I|} \label{eq:d2-lb}
\end{eqnarray}

For a pair of distinct vertices $u,v \in Q$ set $d(u,v)=0$ if $(u,v) \in E(G)$ and otherwise set $d(u,v)$
to be the number of vertices $y \in I$ incident to both $u$ and $v$. Recalling that $I$ is an independent set in $H$, we see that $u,v$
belong to at least ${d(u,v) \choose 2}$ induced copies of $C_4$. Hence,
the number of induced copies of $C_4$ in $H$ is at least
\begin{eqnarray}
	\sum_{u,v \in Q}{\binom{d(u,v)}{2}} &\geq& \binom{|Q|}{2} \cdot \binom{\sum_{u,v \in Q}d(u,v)/\binom{|Q|}{2}}{2}
    \label{eq:Jensen} \\
	&=& \binom{|Q|}{2} \cdot \binom{\sum_{y \in I}d_2(y)/\binom{|Q|}{2}}{2} \label{eq:over}\\
  &\geq& \frac{1}{64} \cdot \frac{(\beta')^4 n^8}{|Q|^2\cdot |I|^2} \label{eq:use-d2-lb} \\
  &\geq& 2\beta n^4 \label{eq:beta'-beta} \;.
\end{eqnarray}
 where Equation~\eqref{eq:Jensen} is due to Jensen's inequality, 
 Equation~\eqref{eq:over} uses $\sum_{u,v \in Q}d(u,v) = \sum_{y \in I}d_2(y)$,
 Equation~\eqref{eq:use-d2-lb} uses Equation~\eqref{eq:d2-lb}, and Equation~\eqref{eq:beta'-beta} follows from the definition of $\beta'$ and the fact that $|Q| + |I| = n$, so that $|Q|\cdot |I| \leq n^2/4$.
As an edge-cover of $2\beta n^4$ instances of $C_4$ is at size
at least $2\leq n^2$, this contradicts the premise that $H$ is $\beta$-close to being $C_4$-free.
\end{proof}

We are now ready to prove Lemma~\ref{lem:everything_for_assaf_and_lior}.
It is proved similarly to Theorem~1 in \cite{GS19}.
\begin{proof}[Proof of Lemma~\ref{lem:everything_for_assaf_and_lior}]
We apply Lemma~\ref{lem:cond_regularity} to $G$ with parameters $\alpha = \beta/4$ and
$\gamma = (\alpha/20)^{4\cdot (4000\alpha^{-6})}\cdot \beta/2$ for $\beta = (\eps/4)^4$.
Since $G$ is $C_4$-free, there is a graph $G^{(1)}$ as stated in Lemma~\ref{lem:cond_regularity}. 
If $|X| < \frac{\eps}{4}n$, then we set $G'$ to be the graph obtained from $G^{(1)}$ by removing all edges incident to
vertices in $X$.
By our assumption, $\Delta(G', G^{(1)}) < \eps/4$.
As $\Delta(G^{(1)}, G) < \eps/4$,
by the triangle inequality $\Delta(G', G) < (\eps/2)$.
Since $Y$ is an independent set in $G^{(1)}$, and hence in $G'$
 we get that $G'$ is the empty graph, and thus fulfills all the conditions of Lemma~\ref{lem:everything_for_assaf_and_lior}.
	
Otherwise $(|X| \geq \frac{\eps}{4}n)$,
let $G^{(2)}$ be the graph obtained from $G^{(1)}$ by doing the following:
for every $1 \leq i < j \leq q$, if $(W_i,W_j)$ is a complete (respectively empty) cut in $G^{(1)}$,  then we turn $(Q_i,Q_j)$ into a complete (respectively empty) cut in $G^{(2)}$. By Item~\ref{it:Wi-Wj} in Lemma~\ref{lem:cond_regularity}, one of these options holds. By Item~\ref{it:sum-Qi-Qj} in Lemma~\ref{lem:cond_regularity},
the number of modifications made is at most $\alpha n^2$.
By Item~\ref{it:G-G'-close} in Lemma~\ref{lem:cond_regularity} we have
\begin{equation}\label{eq:G2-G}
\left| E(G^{(2)}) \SymDif E(G) \right| \leq
	\left| E(G^{(2)}) \SymDif E(G^{(1)}) \right| + \left| E(G^{(1)}) \SymDif E(G) \right| <
	(3\alpha + \gamma)n^2 < \beta n^2\;.
\end{equation}

We claim that 	$G^{(2)}[X \setminus Z]$ is induced $C_4$-free.
Assume, contrary to this claim,
that $G^{(2)}[X \setminus Z]$ contains an induced copy of $C_4$, say on the vertices $v_1,v_2,v_3,v_4$. For $1 \leq s \leq 4$, let $i_{s}$ be such that $v_{s} \in Q_{i_{s}}$.
	It is easy to see that by the definition of $G^{(2)}$, every quadruple
	$(w_1,\dots,w_4) \in W_{i_{1}} \times W_{i_{2}} \times W_{i_{3}} \times W_{i_4}$ spans an induced copy of $C_4$ in the graph $G^{(2)}$. By Item~\ref{it:Wi-Wj} in Lemma~\ref{lem:cond_regularity}, $G^{(2)}$ contains
	$$|W_{i_{1}}| \cdot |W_{i_{2}}| \cdot |W_{i_{3}}| \cdot |W_{i_{4}}| \geq
	(\alpha/20)^{16000\alpha^{-6}}|X|^4 \geq
	(\alpha/20)^{16000\alpha^{-6}} (\varepsilon/4)^4 n^4 = \nolinebreak 2\gamma n^4$$ induced copies of $C_4$.
By Item~\ref{it:G-G'-close} in
Lemma~\ref{lem:cond_regularity}, $G[X \setminus Z]$ and $G'[X \setminus Z]$ differ on less than $\gamma n^2$ edges, each of which can participate in at most $n^2$ induced copies of $C_4$.
This implies that $G$ contains at least $\gamma n^4$ induced copies of $C_4$, contradicting the premise of the current
lemma (that $G$ is $C_4$-free).
	
Since $G$ is induced $C_4$-free and $\Delta(G^{(2)},G) \leq \beta$ by Equation~\eqref{eq:G2-G},
$G^{(2)}$ is $\beta$-close to being induced $C_4$-free. We can thus apply Lemma~\ref{lem:cut_is_close_to_IS_vertices_cliques}
to $H=G^{(2)}$
with $Q = X\setminus Z$
and $I = Y\cup Z$. 
Letting $G^{(3)}$ be the graph $H'$ in the outcome of Lemma~\ref{lem:cut_is_close_to_IS_vertices_cliques},
we have that $G^{(3)}$ is $\beta^{1/4}$-close to $G^{(2)}$, and hence its distance to $G$ is
at most $\beta + \beta^{1/4} \leq (\eps/2)$. Furthermore, $G^{(3)}$
is $C_4$-free, and $E(G^{(2)}[Q])\SymDif E(G^{(3)}[Q]) = \emptyset, E(G^{(2)}[I])\SymDif E(G^{(3)}[I]) = \emptyset$.
Setting $G' = G^{(3)}$,
we get that $G'$ satisfies all requirements of the current lemma.
\end{proof}

\section{Proofs for Section~\ref{sec:chordal}}\label{app:chordal}

In this appendix we prove
Theorem~\ref{thm:chordal-iff-clique-tree}, and Lemmas~\ref{lem:chordal_everything_for_assaf_and_lior} and~\ref{lem:two_clique_partition}.

\subsection{Proof of Theorem~\ref{thm:chordal-iff-clique-tree}} \label{app:subsec:proof_thm+chordal_iff_clque_tree}

The theorem as it is proved in~\cite[Thm. 3.2]{Chordal-Intro} states that a \emph{connected} graph $G$ is chordal if and only if it has a clique tree. The generalization 
to non-connected graphs is also true. We prove it here, as an implication of the theorem 
for connected graphs.

For the first direction,
let $G$ be a chordal graph, and let $\{G_i\}$ be the set of connected components in $G$. 
Since $G$ is chordal, each connected component $G_i$ must be chordal as well.
By~\cite[Thm. 3.2]{Chordal-Intro},  each $G_i$ has a clique tree $T_i$.
Let $\mathcal{T}=\{T_i\}$ 
and let $T_G$ be 
 a tree 
obtained from $\mathcal{T}$ by adding arbitrary edges between the trees in $\mathcal{T}$, so as to connect them to a single tree. As 
every maximal clique in $G$ is a maximal clique in some $G_i$, there is a node in $T_G$ for every maximal clique in $G$. 
To show the $T_G$ is a clique tree of $G$ 
it remains to prove that $C_1\cap C_2 \subset C_3$ for any three maximal cliques $C_1,C_2$ and   $C_3$ such that $C_3$ is on the path between $C_1$ and $C_2$ in $T_G$. If $C_1$ and $C_2$ 
are maximal cliques in the same connected component $G_i$, then the path between them in $T_G$ is in the 
subtree $T_i$ (which is a clique tree of $G_i$) and thus $C_1 \cap C_2 \subset C_3$.
If $C_1$  $C_2$ are maximal cliques in different connected components, then $C_1 \cap C_2 = \emptyset \subset C_3$.

For  the other direction, 
let $G$ be a graph with a clique tree $T_G$. We shall show that $G$ is chordal by proving that each connected component of $G$
has a clique tree, and is hence chordal (implying that $G$ is chordal).
Let $H$ be some connected component in $G$, with maximal cliques $C_1,\dots, C_t$.
Consider the subgraph of $T_G$ induced by $C_1,\dots, C_t$, which we denote by $T_G^H$.
We next show that $T_G^H$ is connected, implying that it is a tree, and hence a clique tree for $H$.

Let $C$ and $C'$ be two maximal cliques in $H$. We show that 
there is a path in $T^H_G$ between them.
Let $v_1\in C$ and $v_\ell \in C'$ be two vertices in $G$.
Since $v_1$ and $v_\ell$ belong to   the same connected component $H$, there is a path between them in $H$.
Let $v_1, v_2, \dots v_\ell$ be such a (simple) path. For each $i \in [\ell-1]$,
let $C_i$ be some maximal clique that 
contains the edge $(v_i,v_{i+1})$,
and denote $C_0=C$ and $C_\ell =C'$.
For any $i\in \{0,\dots,\ell\}$,
the path in $T_G$ between $C_i$ and $C_{i+1}$ must only include cliques that contain $v_{i+1}$, as $v_{i+1} \in C_i \cap C_{i+1}$.
Therefore this path is in $T^H_G$.
By combining all paths in $T_G$ between $C_i$ and $C_{i+1}$ (for $i\in \{0,\dots,\ell\}$),
we get a (non simple) path between $C$ and $C'$ that only contains maximal cliques in $H$.
Hence there is a path between $C$ and $C'$ in  $T^H_G$.
We have thus established that $T^H_G$ is a connected subgraph of $T_G$ (i.e., a subtree), as desired.

\subsection{Proof of Lemma~\ref{lem:chordal_everything_for_assaf_and_lior}}

	As $G$ is chordal, it is also $C_4$-free.
	We apply Lemma~\ref{lem:cond_regularity} to $G$ with parameters $\alpha = \beta/5$ and
	$\gamma = (\alpha/20)^{10/\alpha^3 \cdot (4000\alpha^{-6})}\cdot (\eps'/4)^{10/\alpha^3}/2$ for $\beta = (\eps'/4)^4$.
	Since $G$ is $C_4$-free, there is a graph $G^{(1)}$ as stated in Lemma~\ref{lem:cond_regularity}. 
	As in the  proof of Lemma~\ref{lem:everything_for_assaf_and_lior}, if $|X| < \frac{\eps'}{4}n$, then $G$ is $\eps'$-close to the empty graph, and 
setting $G'$ to be the empty graph completes the proof.
	
	Otherwise ($|X| \geq \frac{\eps'}{4}n$), similarly to the proof of Lemma~\ref{lem:everything_for_assaf_and_lior}, let $G^{(2)}$ be the graph obtained from $G^{(1)}$ by doing the following:
	for every $1 \leq i < j \leq q$, if $(W_i,W_j)$ is a complete (respectively empty) cut in $G^{(1)}$,  then we turn $(Q_i,Q_j)$ into a complete (respectively empty) cut in $G^{(2)}$. By Item~\ref{it:Wi-Wj} in Lemma~\ref{lem:cond_regularity}, one of these options holds. By Item~\ref{it:sum-Qi-Qj} in Lemma~\ref{lem:cond_regularity},
	the number of modifications made is at most $\alpha n^2$.
	By Item~\ref{it:G-G'-close} in Lemma~\ref{lem:cond_regularity} we have
	\begin{equation}\label{eq:chordal-G2-G}
	\left| E(G^{(2)}) \SymDif E(G) \right| \leq
	\left| E(G^{(2)}) \SymDif E(G^{(1)}) \right| + \left| E(G^{(1)}) \SymDif E(G) \right| <
	(3\alpha + \gamma)n^2\;.
	\end{equation}

	
	We claim that 	$G^{(2)}[X \setminus Z]$ is chordal. As we only modified non-homogeneous cuts, every $X_i\setminus Z$ is a clique in $G^{(2)}$,
	so that  $G^{(2)}[X \setminus Z]$ has a $(10/\alpha^3)$-clique-cover, and thus every induced cycle in $G^{(2)}[X \setminus Z]$ is of length at most $10/\alpha^3$.
	Assume contrary to the claim that $G^{(2)}[X \setminus Z]$ has an induced cycle of length $3 < \ell \leq 10/\alpha^3$, whose vertices are denoted $x_1, \dots, x_\ell$, such that $x_i \in W_{a_i}$ where $\{a_i\}_{i=1}^\ell = [\ell]$.
	As the cuts $(W_{a_i}, W_{a_j})$ are homogeneous, and are equal in $G^{(1)}$ and $G^{(2)}$, every subset of $\ell$ vertices $(w_1, \dots w_\ell) \in W_{a_1}  \times \dots  \times W_{a_\ell}$ spans an induced cycle of length $\ell$. By Item~\ref{it:Wi-Wj} in Lemma~\ref{lem:cond_regularity},
	the number of induced cycles of length $\ell$ in $G^{(1)}$ is at least
	$$\Pi_{i\in [\ell]} |W_{a_{i}}| \geq  (\alpha/20)^{\ell \cdot 4000\alpha^{-6}} |X|^\ell \geq (\alpha/20)^{10 \alpha^{-3} \cdot 4000\alpha^{-6}}(\eps'/4)^{10/\alpha^3} n^\ell = 2\gamma n^\ell\;.$$
	By Item~\ref{it:G-G'-close} in
	Lemma~\ref{lem:cond_regularity}, $G[X \setminus Z]$ and $G^{(1)}[X \setminus Z]$  differ on less than $\gamma n^2$ edges, each of which can participate in at most $n^{\ell-2}$ such induced cycles.
	This implies that $G$ contains at least $\gamma n^4$ induced copies of $C_\ell$, contradicting the premise of the current
	lemma (that $G$ is chordal).

	Let $G^{(3)}$ be the graph derived from $G^{(2)}$ by removing all edges that are incident to vertices in $Z$.
By Equation~\eqref{eq:chordal-G2-G} and the bound on the size of $Z$,
$$\Delta(G^{(3)}, G)  \leq \Delta(G^{(3)}, G^{(2)}) + \Delta(G^{(2)}, G) \leq \alpha + (3\alpha + \gamma) < \beta\;.$$
	Since $G$ is chordal and $\Delta(G^{(3)},G) \leq \beta$,
we can  apply Lemma~\ref{lem:cut_is_close_to_IS_vertices_cliques}
	to $H=G^{(3)}$
	with $\calJ$ as the family of cycles of size greater than 3, $Q = X\setminus Z$
	and $I = Y\cup Z$. 
	Letting $G^{(4)}$ be the graph $H'$ in the outcome of  Lemma~\ref{lem:cut_is_close_to_IS_vertices_cliques},
	we have that $G^{(4)}$ is $2\beta^{1/4}$-close to $G^{(3)}$, and hence its distance to $G$ is
	at most $\beta + 2\beta^{1/4} \leq (\eps'/2)$. Furthermore, $G^{(4)}$
	is Chordal, $E(G^{(3)}[Q])\SymDif E(G^{(4)}[Q]) = \emptyset$ and $E(G^{(3)}[I])\SymDif E(G^{(4)}[I]) = \emptyset$.
	Setting $G' = G^{(4)}$,
	we get that $G'$ satisfies all requirements of 
Lemma~\ref{lem:chordal_everything_for_assaf_and_lior}.


\subsection{Proof of Lemma~\ref{lem:two_clique_partition}}

To prove 
Lemma~\ref{lem:two_clique_partition}
we make use of the following definitions. 
\begin{definition}\label{def:M2-free}
We say that a cut $(X,Y)$ in a graph $G = (V,E)$ is {\sf induced $M_2$-free} if there are no four vertices $x,x',y,y'$ such that $x,x' \in X$, $y,y'\in Y$,
$(x,y), (x',y') \in E$ and $(x,y'), (x',y) \notin E$.
\end{definition}
Observe that if $X$ and $Y$ are cliques, then $G[X\cup Y]$ is induced $C_4$-free if and only if $(X,Y)$ is induced $M_2$-free.

\begin{definition}\label{def:N-S-T}
	Let $E'$ be an edge set, $S$ a vertex set and $u$ a vertex. Then
	$N_{(S, E')}(u) = \{v \,|\, v \in S,\, (u,v) \in E', \}$ is the set of neighbors of $u$ in $S$ in the graph $(S'=\{u\}\cup S, E')$, and for a set of vertices $T$,
	$N_{(S, E')}(T) = \{v \,|\, v \in S,\, \forall u\in T \; \; (u,v) \in E'\}$ is the set of vertices in $S$ that neighbor all vertices in  $T$.
	When the edge set $E'$ is clear from context, we simply write $N_S(v)$ or $N_S(T)$.
\end{definition}

The next claim is a slight variant of Lemma 2.2 in~\cite{GS19}, but its proof is exactly the same as the proof of that lemma.
\begin{claim}[{\cite[Lem. 2.2]{GS19}}] \label{clm:two-partition}
	If $(X,Y)$ is induced $M_2$-free, then for every integer $r\geq 1$ there   are partitions
$(X_1,\dots,X_r)$ and $(Y_1,\dots,Y_{r+1})$ of $X$ and $Y$ respectively,
 such that $|X_i| = \frac{|X|}{r}$ for every $i \in [r]$ and the cut $(X_i, Y_j)$ is complete for each $i>j$, and empty for each $i<j$.
\end{claim}

\medskip
\begin{proof}[Proof of Lemma~\ref{lem:two_clique_partition}]
	Let $X_1, \dots, X_r$ and $Y_1, \dots, Y_{r+1}$ be as defined in Claim~\ref{clm:two-partition} for $X=C_1\setminus C_2$, $Y=C_2\setminus C_1$, and $r = 1/\eps'$.
Let $G'=(V,E')$ be the  graph obtained from $G=(V,E)$ by removing, for each $i \in [r]$, all edges in the cut $(X_i, Y_i)$ if the cut is not 
homogeneous. Note that in $G'$, for each pair $i,j$ the cut $(X_i, Y_j)$ is homogeneous. Also note that the cut $(X, Y)$ is still $M_2$-free in $G'$ (as assuming the contrary implies an induced $C_4$ in $G$, contradicting the premise of the current lemma that $G$ is chordal). In what follows we verify that the conditions in the three items of Definition~\ref{def:simplification} hold.

\smallskip
	Starting with Item~\ref{item:G-Gp-close}, to show that $\Delta(G, G') \leq \eps'$ observe that
$$\sum_{(X_i,Y_i)} |\{(x,y) \; | x\in X_i, y\in Y_i\} | \leq \sum_{(X_i,Y_i)}|X_i|\cdot |Y_i| \leq \frac{|X|}{r} \cdot |Y|  \leq \frac{n^2}{r} = \eps' n^2\;.$$

\smallskip
We now turn to Item~\ref{item:short-path}. In what follows, the underlying graph is $G'$, so that when we say that a cut is homogeneous, we mean in $G'$, and  we use the notation $N_Y(\cdot)$ as a shorthand for $N_{Y,E'}(\cdot)$.
	Let us partition $X$ into equivalence classes 
$H_1,\dots,H_k$ such that all vertices in the same class have the same set of neighbors in $Y$.
By Claim~\ref{clm:two-partition}, for each cut $(X_i, Y_j)$, the cut is complete if $i>j$ and empty if $i<j$.
By the definition of $G'$, if $i=j$, then the cut is either complete or empty.
Hence each $H_t$ is the union of $X_{i_t+1},X_{i_t+1},\dots,X_{i_{t+1}}$ for some $0 \leq i_t < i_{t+1}\leq r$,
and since
all $H_t$ are disjoint, $k \leq r$.
We also denote 
$H_{k+1} = C_1 \cap C_2$. Note that $N_Y(H_{k+1}) = Y$ and $\bigcup_{i=1}^{k+1} H_i = C_1$.

We define a set of cliques $P_1 \dots P_{k+1}$ as follows:
$$P_i = \left( \bigcup_{j\geq i} H_{j}\right) \cup N_Y(H_{i})\;.$$
	Let $\calQ = \{P_1, \dots, P_{k+1}\}$. We proceed to show that $\calQ$ is the set of maximal cliques in $G'[C_1 \cup C_2]$.

	Let $\wtC$ be some clique in $G'[C_1 \cup C_2]$. We show that there exists a clique in $\calQ$ that is a superset of $\wtC$.
As $P_i$ includes all vertices in $X$ that are neighbors of $N_Y(H_i)$, we get that $\wtC \cap X \subset P_i$.
Let $\tilde{x} \in H_i$ be 
a vertex in $\wtQ$ with the largest number of neighbors in $Y$.
	Observe that the set $\wtC\cap Y$ cannot contain a vertex that is not in $N_Y(\tilde{x})$ as it is a clique. By definition of $P_i$, $N_Y(\tilde{x}) = N_Y(P_i)$. Therefore, $\wtC \cap Y \subset P_i$. Also note that for any $i$, $C_1\cap C_2 \subseteq P_i$. As $\wtC \subset (X\cup Y \cup (C_1\cap C_2))$, we have that $\wtC \subseteq P_i$.

	To show that every clique in $\calQ$ is maximal, consider any clique $P_i \in \calQ$ and vertex $v\in C_1\cup C_2$  such that $v \notin P_i$. We show that $v$ has some non-neighbor in $P_i$. By definition of $P_i$, if $v \in X$, then there exists
$i' < i$
such that $v\in H_{i'}$ and thus it cannot be a neighbor of the vertices in
$N_Y(H_i) \setminus N_Y(H_{i'}) \subset P_i$.
If $v \in Y$, as $N_Y(H_i)\subset P_i$, it cannot be a neighbor of all vertices in $P_i \cap H_i$.

	We are now ready to show that the path $P_1, \dots, P_{k+1}$ is a clique tree 
for $G'[C_1 \cup C_2]$. Let $i<j<\ell$. We show that $P_i \cap P_\ell \subset P_j$.
As $P_i \cap P_\ell \cap X = H_\ell \cup \dots \cup H_{k+1}$, and
$P_j \cap X = H_j \cup \dots \cup H_{k+1}$,
we have that
$P_i \cap P_\ell \cap X \subset P_j \cap X$.
As $P_i \cap P_\ell \cap Y = N_Y(H_i)$, and
$P_j \cap Y = N_Y(H_j)$,
we have that
$P_i \cap P_\ell \cap Y \subset P_j \cap Y$.
Also note that all three cliques contain $C_1 \cap C_2$, and that $P_i, P_j, P_\ell \subset (X\cup Y \cup (C_1\cap C_2))$. Therefore, combining the above equations, we get that $P_i \cap P_\ell \subset P_j$, as required.

\smallskip
It remains to verify that the condition in Item~\ref{item:close-clique} of Definition~\ref{def:simplification} holds.	
Namely, we must show that for every clique $C$ in $G[C_1 \cup C_2]$, there exists a maximal clique $C'$ in $G'[C_1 \cap C_2]$ such that $|C\setminus C'| \leq \eps'\cdot n$. For such a clique $C$,
	let $x$ be the vertex in $C \cap X$ with the smallest number of neighbors in $Y$ (with respect to the edge set $E$).
Since $G$ is chordal, and both $X$ and $Y$ are cliques, the cut $(X,Y)$ is $M_2$-free.
	Let $i\in [r]$ be such that $x \in X_i$, and let $j\in [k]$ be the 
largest index such that $P_j$ contains all vertices in 
$\bigcup_{i'>i} X_{i'}$.
	We show that $|C \setminus P_j| \leq |X_{i}| \leq \eps'\cdot n$.
By the definition of $P_j$ (based on the $H_t$s) we have that $|(C \setminus P_j)\cap X| \leq |X_{i}|$. Also, as $C_1\cap C_2 \subset P_j$, we have that $|(C\setminus P_j) \cap (C_1\cap C_2)|=0$.
It remains to show that $|(C \setminus P_j) \cap Y| = 0$.
Let $y \in N_{(Y, E)}(x)$ and let $\ell\in [r+1]$ be such that $y\in Y_\ell$.
	If $i \neq \ell$, then the cut $(X_{i}, Y_\ell)$ is homogeneous in $G$. 
	By Claim~\ref{clm:two-partition}, for all $i' \leq i$ and for all $x' \in X_{i'}$ and $y'\in Y_\ell$,
we have that $y' \in N_{(Y,E)}(x')$. Therefore, the cuts $(X_{i'}, Y_\ell)$ are homogeneous in $G$, and thus every edge in them is also an edge in $G'$.
Therefore, $y \in N_{(Y, E')}(P_j\cap X)$.
	
	If $i = \ell$, then by
	Claim~\ref{clm:two-partition}, for all $i' < i$ and for all $x' \in X_{i'}$, $y \in N_{(Y,E)}(x') $. As $i' \neq i$, the cuts $(X_{i'}, Y_\ell)$ are homogeneous in $G$, and thus every edge in them is also an edge in $G'$. Therefore, $y \in N_{(Y, E')}(P_j\cap X)$.
	
	We have thus established that $N_{(Y, E)}(x) \setminus N_{(Y,E)}(H_j) = \emptyset$, and the proof of the current lemma is completed.
\end{proof}

\end{document}